\documentclass[11pt]{amsart}

\usepackage{amssymb,amsmath,color,hyperref}
\usepackage[mathscr]{eucal}

\theoremstyle{plain}
\newtheorem{thm}{Theorem}[section]
\newtheorem{theorem}[thm]{Theorem}
\newtheorem{lem}[thm]{Lemma}

\newtheorem{cor}[thm]{Corollary}
\newtheorem{prop}[thm]{Proposition}

\theoremstyle{definition}

\theoremstyle{remark}

\newtheorem*{remark*}{Remark}

\numberwithin{equation}{section}

        \newcommand{\field}[1]{{\mathbb{#1}}}
        \newcommand{\NN}{\field{N}}
        \newcommand{\ZZ}{\field{Z}}
        
        \newcommand{\RR}{\field{R}}
        \newcommand{\CC}{\field{C}}

\allowdisplaybreaks

\begin{document}

\title[Semiclassical asymptotic expansions]{Semiclassical asymptotic expansions for functions of the Bochner-Schr\"odinger operator}

\author[Y. A. Kordyukov]{Yuri A. Kordyukov}
\address{Institute of Mathematics, Ufa Federal Research Centre, Russian Academy of Sciences, 112~Chernyshevsky str., 450008 Ufa, Russia} \email{yurikor@matem.anrb.ru}

%\thanks{Partially supported by the development program of the Regional Scientific and Educational Mathematical Center of the Volga Federal District, agreement N 075-02-2020-1478.}



\begin{abstract}
The Bochner-Schr\"odinger operator $H_{p}=\frac 1p\Delta^{L^p\otimes E}+V$ on tensor powers $L^p$ of a Hermitian line bundle $L$ twisted by a Hermitian vector bundle $E$ on a Riemannian manifold of bounded geometry is studied. For any function $\varphi\in \mathcal S(\mathbb R)$, we consider the bounded linear operator $\varphi(H_p)$ in $L^2(X,L^p\otimes E)$ defined by the spectral theorem and describe an asymptotic expansion of its smooth Schwartz kernel in a fixed neighborhood of the diagonal in the semiclassical limit $p\to \infty$. In particular, we prove that the trace of the operator $\varphi(H_p)$ admits a complete asymptotic expansion in powers of $p^{-1/2}$ as $p\to \infty$. We also prove a result on the asymptotic  localization of the Schwartz kernel of the spectral projection on the diagonal in the case when the curvature is of full rank. 
\end{abstract}

\date{\today}

 \maketitle
%\tableofcontents
\section{Introduction}
Let $(X,g)$ be a complete Riemannian manifold of dimension $d$, $(L,h^L)$ a Hermitian line bundle on $X$ with a Hermitian connection $\nabla^L$ and $(E,h^E)$ a Hermitian vector bundle on $X$ with a Hermitian connection $\nabla^E$. We suppose that $(X, g)$ is a Riemannian manifold of bounded geometry and the bundles $L$ and $E$ have bounded geometry. This means that the curvatures $R^{TX}$, $R^L$ and $R^E$ of the Levi--Civita connection $\nabla^{TX}$, connections $\nabla^L$ and $\nabla^E$, respectively, and their covariant derivatives of any order are uniformly bounded on $X$ in the norm induced by $g$, $h^L$ and $h^E$, and the injectivity radius $r_X$ of $(X, g)$ is positive. 

For any $p\in \NN$, let $L^p:=L^{\otimes p}$ be the $p$th tensor power of $L$ and let
\[
\nabla^{L^p\otimes E}: {C}^\infty(X,L^p\otimes E)\to
{C}^\infty(X, T^*X \otimes L^p\otimes E)
\] 
be the Hermitian connection on $L^p\otimes E$ induced by $\nabla^{L}$ and $\nabla^E$. Consider the induced Bochner Laplacian $\Delta^{L^p\otimes E}$ acting on $C^\infty(X,L^p\otimes E)$ by
\begin{equation}\label{e:def-Bochner}
\Delta^{L^p\otimes E}=\big(\nabla^{L^p\otimes E}\big)^{\!*}\,
\nabla^{L^p\otimes E},
\end{equation} 
where $\big(\nabla^{L^p\otimes E}\big)^{\!*}: {C}^\infty(X,T^*X\otimes L^p\otimes E)\to
{C}^\infty(X,L^p\otimes E)$ is the formal adjoint of  $\nabla^{L^p\otimes E}$. Let $V\in C^\infty(X,\operatorname{End}(E))$ be a self-adjoint endomorphism of $E$. The Bochner-Schr\"odinger operator $H_p$ is a second order differential operator acting on $C^\infty(X,L^p\otimes E)$ by
\[
H_{p}=\frac 1p\Delta^{L^p\otimes E}+V. 
\] 

Such operators have been introduced and studied by Demailly in connection with the holomorphic Morse inequalities for the Dolbeault cohomology associated with high tensor powers of a holomorphic Hermitian line bundle over a compact complex manifold \cite{Demailly85} (see also \cite{bismut87,Demailly91,ma-ma:book} and references therein). We will treat the parameter $\hbar = \frac 1p$ as the semiclassical parameter and  the limit $p\to \infty$ as the semiclassical limit, as well known, for instance, in geometric quantization. When the Hermitian line bundle $(L,h^L)$ is trivial and $(E,h^E)$ is the trivial Hermitian line bundle with the trivial connection $\nabla^E$, the operator $H_p$ is closely related with the semiclassical magnetic Schr\"odinger operator  (see, for instance,  \cite[Remark 2]{higherLL}). 

Since $(X,g)$ is complete, the operator $H_p$ is essentially self-adjoint  in the Hilbert space $L^2(X,L^p\otimes E)$ with initial domain  $C^\infty_c(X,L^p\otimes E)$, see \cite[Theorem 2.4]{ko-ma-ma}. We still denote by $H_p$ its unique self-adjoint extension. In \cite{higherLL}. a rough asymptotic description of its spectrum in the semiclassical limit is given (see also \cite{charles21,FT} for the case of compact manifolds). For any function $\varphi\in \mathcal S(\RR)$, one can define by the spectral theorem a bounded linear operator $\varphi(H_p)$ in $L^2(X,L^p\otimes E)$ with smooth Schwartz kernel.  Let $K_{\varphi(H_{p})}\in {C}^{\infty}(X\times X, \pi_1^*(L^p\otimes E)\otimes \pi_2^*(L^p\otimes E)^*)$ be the Schwartz kernel of the operator $\varphi(H_{p})$ with respect to the Riemannian volume form $dv_X$.  (Here $\pi_1,\pi_2 : X\times X\to X$ are the natural projections.) Using the finite propagation speed property of solutions of hyperbolic equations, one can show that for any $\varepsilon >0$ and $k=0,1,2,\ldots,$
\[
\left|K_{\varphi(H_{p})}(x,x^\prime)\right|_{{C}^k}=\mathcal O(p^{-\infty}), \quad d(x,x^\prime)>\varepsilon.
\]
Here $d(x,x^\prime)$ is the geodesic distance and $|K_{\varphi(H_{p})}(x,x^\prime)|_{{C}^k}$ denotes the pointwise ${C}^k$-seminorm of the section $K_{\varphi(H_{p})}$ at a point $(x, x^\prime)\in X\times X$, which is the sum of the norms induced by $h^L, h^E$ and $g^{TX}$ of the derivatives up to order $k$ of $K_{\varphi(H_{p})}$ with respect to the connection $\nabla^{L^p\otimes E}$ and the Levi-Civita connection $\nabla^{TX}$ evaluated at $(x, x^\prime)$.

In this paper, we prove a complete asymptotic expansion for  $K_{\varphi(H_{p})}$ as $p\to \infty$ in a fixed neighborhood of the diagonal (independent of $p$). Such kind of expansion is called the full off-diagonal expansion following Ma-Marinescu's book \cite[Chapter 4]{ma-ma:book}.

First, we introduce normal coordinates near an arbitrary point $x_0\in X$. 
We denote by $B^{X}(x_0,r)$ and $B^{T_{x_0}X}(0,r)$ the open balls in $X$ and $T_{x_0}X$ with center $x_0$ and radius $r$, respectively. We identify $B^{T_{x_0}X}(0,r_X)$ with $B^{X}(x_0,r_X)$ via the exponential map $\exp^X_{x_0}: T_{x_0}X \to X$. Furthermore, we choose trivializations of the bundles $L$ and $E$ over $B^{X}(x_0,r_X)$,   identifying their fibers $L_Z$ and $E_Z$ at $Z\in B^{T_{x_0}X}(0,r_X)\cong B^{X}(x_0,r_X)$ with the spaces  $L_{x_0}$ and $E_{x_0}$ by parallel transport with respect to the connections $\nabla^L$ and $\nabla^E$ along the curve $\gamma_Z : [0,1]\ni u \to \exp^X_{x_0}(uZ)$.

Consider the fiberwise product $TX\times_X TX=\{(Z,Z^\prime)\in T_{x_0}X\times T_{x_0}X : x_0\in X\}$. Let $\pi : TX\times_X TX\to X$ be the natural projection given by  $\pi(Z,Z^\prime)=x_0$. The kernel $K_{\varphi(H_{p})}$ induces a smooth section $K_{\varphi(H_{p}),x_0}(Z,Z^\prime)$ of the vector bundle $\pi^*(\operatorname{End}(E))$ on $TX\times_X TX$ defined for all $x_0\in X$ and $Z,Z^\prime\in T_{x_0}X$ with $|Z|, |Z^\prime|<r_X$:
\[
K_{\varphi(H_p),x_0}(Z,Z^\prime)=K_{\varphi(H_p)}(\exp^X_{x_0}(Z),\exp^X_{x_0}(Z^\prime)).
\]

Let $dv_{X,x_0}$ denote the Riemannian volume form of the Euclidean space $(T_{x_0}X, g_{x_0})$. We define a smooth function $\kappa_{x_0}$ on $B^{T_{x_0}X}(0,r_X)\cong B^{X}(x_0,r_X)$ by the equation
\begin{equation}\label{e:def-kappa}
dv_{X}(Z)=\kappa_{x_0}(Z)dv_{X,x_0}(Z), \quad Z\in B^{T_{x_0}X}(0,r_X). 
\end{equation}

The main result of the paper is the following theorem. 

\begin{theorem}\label{t:main}
For any $j,m,m^\prime\in \mathbb N$, there exists $M\in \NN$ such that, for any $N\in \NN$, there exists $C>0$ such that for any $p\geq 1$, $x_0\in X$ and $Z,Z^\prime\in T_{x_0}X$, $|Z|, |Z^\prime|<r_X$, 
\begin{multline}\label{e:main-exp}
\sup_{|\alpha|+|\alpha^\prime|\leq m}\Bigg|\frac{\partial^{|\alpha|+|\alpha^\prime|}}{\partial Z^\alpha\partial Z^{\prime\alpha^\prime}}\Bigg(p^{-\frac d2}K_{\varphi(H_{p}),x_0}(Z,Z^\prime)\\
-\sum_{r=0}^jF_{r,x_0}(\sqrt{p} Z, \sqrt{p}Z^\prime)\kappa_{x_0}^{-\frac 12}(Z)\kappa_{x_0}^{-\frac 12}(Z^\prime)p^{-\frac{r}{2}}\Bigg) \Bigg|_{C^{m^\prime}_b(X)}\\
\leq Cp^{-\frac{j-m+1}{2}}(1+\sqrt{p}|Z|+\sqrt{p}|Z^\prime|)^{M}(1+\sqrt{p}|Z-Z^\prime|)^{-N}+\mathcal O(p^{-\infty}), 
\end{multline}
where $F_{r,x_0}(Z,Z^\prime),$ $r=0,1,\ldots$, is a smooth section of the vector bundle $\pi^*(\operatorname{End}(E))$ on $TX\times_X TX$.
\end{theorem}
Here $C^{m^\prime}_b(X)$ is the $C^{m^\prime}_b$-norm for the parameter $x_0\in X$. 

Putting $Z=Z^\prime=0$, we get on-diagonal expanion for $K_{\varphi(H_{p})}$. 

\begin{cor}\label{c:main} 
For any $j\in \mathbb N$, there exists $C>0$ such that for any $p\geq 1$ and $x_0\in X$,  
\[
\Bigg|p^{-\frac d2}K_{\varphi(H_{p})} (x_0,x_0)-\sum_{r=0}^jf_{r}(x_0)p^{-\frac{r}{2}}\Bigg|\leq Cp^{-\frac{j+1}{2}},
\]
where $f_{r},$ $r=0,1,\ldots$, is a smooth section of the vector bundle $\operatorname{End}(E)$ on $X$ given by
\[
f_{r}(x_0)=F_{r,x_0}(0,0).
\]
\end{cor}

As an immediate consequence of Corollary \ref{c:main}, we get the following semiclassical trace formula for the operator $H_p$ in the case when $X$ is compact. 

\begin{cor}\label{t:trace} There exists a sequence of distributions $f_r \in \mathcal D^\prime(\mathbb R), r=0,1,\ldots$, such that for any $\varphi\in C^\infty_c(\mathbb R)$, we have an asymptotic expansion
\begin{equation}\label{e:trace}
\operatorname{tr} \varphi(H_{p})\sim p^{\frac d2}\sum_{r=0}^{j}\langle f_{r}, \varphi\rangle p^{-\frac{r}{2}}, \quad p \to \infty,
\end{equation}
which means that for any $j\in \mathbb N$ there exists $C>0$ such that 
\[
\left|p^{-\frac d2} \operatorname{tr} \varphi(H_{p})-\sum_{r=0}^{j}\langle f_{r}, \varphi\rangle p^{-\frac{r}{2}}\right|\leq Cp^{-\frac{j+1}{2}}, \quad p \in \mathbb N.
\]
\end{cor}

Explicit formulas for the leading coefficients $F_{0,x_0}$ and $f_{0}$ involve particular second order differential operators  (the model operators) associated with an arbitrary point $x_0\in X$, introduced already  by Demailly \cite{Demailly85,Demailly91}. They are obtained from $H_p$ by freezing its coefficients at $x_0$.

Let $x_0\in X$. Consider the trivial Hermitian vector bundle $E_0$ over $T_{x_0}X$ with the fiber $E_{x_0}$. We introduce the connection 
\begin{equation}\label{e:nablaL0}
\nabla^{(x_0)}_{v}=\nabla_{v}+\frac{1}{2}R^L_{x_0}(w,v), \quad v\in T_w(T_{x_0}X), \quad w\in T_{x_0}X,
\end{equation}
acting on $C^\infty(T_{x_0}X, E_0)\cong C^\infty(T_{x_0}X, E_{x_0})$, where $\nabla_{v}$ denotes the usual derivative of a vector-valued function along $v$ and  $R^L=(\nabla^L)^2$ is the curvature of the connection $\nabla^L $.  Its curvature is constant and equals $R^L_{x_0}$. Denote by $\Delta^{(x_0)}$ the associated Bochner Laplacian.

The model operator $\mathcal H^{(x_0)}$ on $C^\infty(T_{x_0}X, E_{x_0})$ is defined as 
\begin{equation}\label{e:DeltaL0p}
\mathcal H^{(x_0)}=\Delta^{(x_0)}+V(x_0).
\end{equation}

\begin{thm}\label{t:leading-coefficient}
The leading coefficient $F_{0,x_0}$ in the asymptotic expansion \eqref{e:main-exp} is the Schwartz kernel $K_{\varphi(\mathcal H^{(x_0)})}$ of the corresponding function $\varphi(\mathcal H^{(x_0)})$ of the model operator $\mathcal H^{(x_0)}$:
\begin{equation}\label{e:F0}
F_{0,x_0}(Z,Z^\prime)=K_{\varphi(\mathcal H^{(x_0)})}(Z,Z^\prime).
\end{equation}
as a consequence, we get 
\begin{equation}\label{e:f0}
f_{0}(x_0)=K_{\varphi(\mathcal H^{(x_0)})}(0,0).
\end{equation}
\end{thm}

One can compute the Schwartz kernel $K_{\varphi(\mathcal H^{(x_0)})}$ and get a more explicit formula for $f_{0}(x_0)$.  Consider the real-valued closed 2-form $\mathbf B$ given by 
\begin{equation}\label{e:def-omega}
\mathbf B=iR^L. 
\end{equation} 
For $x_0\in X$, let $B_{x_0} : T_{x_0}X\to T_{x_0}X$ be a skew-adjoint operator such that 
\[
\mathbf B_{x_0}(u,v)=g(B_{x_0}u,v), \quad u,v\in T_{x_0}X. 
\]
Suppose that the rank of $B_{x_0}$ equals $2n=2n_{x_0}$. Its non-zero eigenvalues have the form $\pm i a_j({x_0}), j=1,\ldots,n,$ with $a_j({x_0})>0$ and, if $d>2n$, zero is an eigenvalue of multiplicity $d-2n$. Denote by $V_\mu({x_0}), \mu=1,\ldots,\operatorname{rank}(E)$, the eigenvalues of $V({x_0})$. 

For $\mathbf k \in \mathbb Z^n_+$ and $\mu=1,\ldots,\operatorname{rank}(E)$, set
\begin{equation}\label{e:def-Lambda}
\Lambda_{\mathbf k,\mu}(x_0)=\sum_{j=1}^n(2k_j+1) a_j(x_0)+V_\mu(x_0).
\end{equation}
%In the full-rank case $d=2n$, the spectrum of $\mathcal H^{(x_0)}$ is a countable set of eigenvalues of infinite multiplicity:
%\[
%\sigma(\mathcal H^{(x_0)})=\left\{\Lambda_{\mathbf k,\mu}({x_0})\,:\, \mathbf k=(k_1,\cdots,k_n)\in\ZZ_+^n, \mu=1,\ldots,r\right\}. 
%\]
%If $d>2n$, the spectrum of $\mathcal H^{(x_0)}$ is a semiaxis:
%\[
%\sigma(\mathcal H^{(x_0)})=[\Lambda_0(x_0), +\infty),
%\]
%where 
%\[
%\Lambda_0(x_0):=\sum_{j=1}^n a_j(x_0)+\min _\mu V_\mu(x_0). 
%\]
Then, in the full-rank case $d=2n$, we have
\begin{equation}\label{e:f0-2n}
f_{0}(x_0)=\frac{1}{(2\pi)^{n}} \left(\prod_{j=1}^n a_j(x_0)\right)  \sum_{\mathbf k, \mu}\varphi(\Lambda_{\mathbf k,\mu}(x_0)), 
\end{equation}
and, for $d>2n$,
\begin{equation}\label{e:f0-d}
f_{0}(x_0)=\frac{1}{(2\pi)^{d-n}} \left(\prod_{j=1}^n a_j(x_0)\right)  \sum_{\mathbf k, \mu}\int_{\RR^{d-2n}} \varphi(|\xi|^2+\Lambda_{\mathbf k,\mu}(x_0))d\xi.
\end{equation}
We refer the reader to Section \ref{s:computation} for a description of the structure of the lower order coefficients (see \eqref{e:fr1-form} and \eqref{e:fr2-form}).

As an immediate consequence of Corollary \ref{c:main} and the description of the coefficients of the asymptotic expanion given in Section \ref{s:computation}, we obtain the following result on the asymptotic  localization of the Schwartz kernel of the spectral projection on the diagonal in the case when the curvature is of full rank, which is an improvement of \cite[Theorem 1.3]{charles21}.

\begin{thm}\label{t:local}
Assume that, for some $x_0\in X$, the rank of $B_{x_0}$ equals $d$ and an interval $(\alpha,\beta)$ does not contain any $\Lambda_{\mathbf k,\mu}(x_0)$ with  $\mathbf k \in \mathbb Z^n_+$ and $\mu=1,\ldots,\operatorname{rank}(E)$. For any $\varphi\in \mathcal S(\RR)$ such that ${\rm supp}\,\varphi \subset (\alpha,\beta)$, 
\begin{equation} \label{e:loc1}
\left|K_{\varphi(H_{p})}(x_0,x_0)\right|_{{C}^k}=\mathcal O(p^{-\infty}), \quad k=0,1,\ldots, \quad p\to \infty.
\end{equation}
Moreover, if an interval $[\alpha,\beta]$ does not contain any $\Lambda_{\mathbf k,\mu}(x_0)$ with  $\mathbf k \in \mathbb Z^n_+$ and $\mu=1,\ldots,\operatorname{rank}(E)$, then the Schwartz kernel of the spectral projection $E_{[\alpha,\beta]}$ of the operator $H_p$ associated with $[\alpha,\beta]$ satisfies 
\begin{equation} \label{e:loc2}
\left|E_{[\alpha,\beta]}(x_0,x_0)\right|=\mathcal O(p^{-\infty}),\quad p\to \infty.
\end{equation}
\end{thm}

If the manifold $X$ is compact, then the operator $H_p$ has discrete spectrum, and there exists an orthonormal basis $\{u_{p,j}\in C^\infty(X,L^p\otimes E), j\in \mathbb N\}$ with the corresponding eigenvalues $\lambda_{p,j}$:
\[
H_pu_{p,j} = \lambda_{p,j}u_{p,j}. 
\]
Then 
\[
E_{[\alpha,\beta]}(x_0,x_0)=\sum_{j:\lambda_{p,j}\in [\alpha,\beta]} |u_{p,j}(x_0)|^2. 
\]
By Theorem~\ref{t:local}, we conclude that, if, for some $x_0\in X$, an interval $[\alpha,\beta]$ does not contains any $\Lambda_{\mathbf k,\mu}(x_0)$ with  $\mathbf k \in \mathbb Z^n_+$ and $\mu=1,\ldots,\operatorname{rank}(E)$, then, for any sequence $\{u_{p}\in C^\infty(X,L^p\otimes E), p\in \mathbb N\}$ of eigenfunctions of $H_p$ with the corresponding eigenvalues $\lambda_{p}$ in $[\alpha,\beta]$ for any $p\in \mathbb N$, we have
\[
 |u_{p}(x_0)|=\mathcal O(p^{-\infty}), \quad p\to \infty.
\]
In other words, the essential support of the sequence $\{u_{p}, p\in \mathbb N\}$ is contained in the set of all $x_0\in X$ such that $\Lambda_{\mathbf k,\mu}(x_0)\in [\alpha,\beta]$ for some $\mathbf k \in \mathbb Z^n_+$ and $\mu=1,\ldots,\operatorname{rank}(E)$. This result complements the asymptotic description of the spectrum of $H_p$ in terms of $\Lambda_{\mathbf k,\mu}(x_0)$ given in \cite{charles21,higherLL}.

Theorem~\ref{t:main} is the analog of Theorem 4.18' in \cite{dai-liu-ma} and Theorem 1 in \cite{Kor18} on the full off-diagonal expasions for the (generalized) Bergman kernels. Its proof is based, first of all, on the approach to the asymptotic analysis of generalized Bergman kernels developed in \cite{dai-liu-ma,ma-ma:book,ma-ma08} and inspired by Bismut-Lebeau localization technique \cite{BL}. In addition, we use the functional calculus based on the Helffer-Sjostrand formula \cite{HS-LNP345}, that allows us to study more general functions of the operator. Such a strategy has been applied in a close setting in \cite{Savale17} and \cite{Marinescu-Savale}. Note that it does not require the spectral gap property for the operator $H_p$ as in \cite{dai-liu-ma,ma-ma:book,ma-ma08} and, therefore, allows us to drop the assumption on the curvature $R^L$ to be non-degenerate (and even constant-rank). Finally, we use the methods of \cite{Kor18} based on weighted estimates to get the full off-diagonal expansions (see also \cite{ko-ma-ma,higherLL} for the extension to manifolds of bounded geometry). In the case of Bergman kernels, the remainder estimates in asymptotic expansions are exponentially decreasing for large $|Z-Z^\prime|$. It is clear that, for general functions from the Schwartz class,  these remainder estimates should be only rapidly decreasing. 

As mentioned above, the operator $H_p$ was first studied by Demailly in \cite{Demailly85,Demailly91}. In particular, Demailly proved the Weyl asymptotic formula for the eigenvalue distribution function of the operator  $H_p$ without any restrictions on the curvature of $L$. The heat kernel proof of the Weyl asymptotic formule was given in \cite{bouche90}. We refer to \cite{bismut87,bouche90,Demailly91,ma-ma:book,ma-ma-zelditch15} and references therein for the study of the heat kernel associated with $H_p$ and to \cite{Morin19,charles21} for the study of the Weyl asymptotic formula in the case when the curvature $R^L$ is nondegenerate. 

In the case when the Hermitian line bundle $(L,h^L)$ is trivial and $(E,h^E)$ is a trivial Hermitian line bundle with a trivial connection $\nabla^E$, the formula \eqref{e:trace} is a particular case of the Gutzwiller trace formula for the semiclassical magnetic Schr\"odinger operator and the zero energy level of the classical Hamiltonian, which is critical in this case. In \cite{K-D97},  a similar formula has been obtained for the Schr\"odinger operator $h^2\Delta+V$ in $\mathbb R^n$ and a non-degenerate critical energy level (see also \cite{BPU95,camus04} for a related study). 

In the case when the curvature $R^L$ is nondegenerate and the Fourier transform of $\varphi$ is compactly supported, similar asymptotic expansions for the operators $\varphi(H_p)$ are proved in \cite{B-Uribe07}, using the theory of Fourier integral operators of Hermite type. 

We refer the interested reader to \cite{UMN-trace} for a survey of basic notions and results related with the trace formula \eqref{e:trace}.

 \section{Localization of the problem}\label{local}
In this section, we localize the problem, using the constructions of \cite[Sections 1.1 and 1.2]{ma-ma08}. 

First, we will use normal coordinates near an arbitrary point $x_0\in X$ and trivializations of vector bundles $L$ and $E$ defined in Introduction. 
Consider the trivial bundles $L_0$ and $E_0$  on $T_{x_0}X$ with fibers $L_{x_0}$ and $E_{x_0}$, respectively. The above identifications induce the Riemannian metric $g$ on $B^{T_{x_0}X}(0,\varepsilon)$
as the connections $\nabla^L$ and $\nabla^E$ and the Hermitian metrics $h^L$ and $h^E$ on the restrictions of $L_0$ and $E_0$ to $B^{T_{x_0}X}(0,\varepsilon)$.

Now we fix $\varepsilon \in (0,r_X)$ and we extend these geometric objects from $B^{T_{x_0}X}(0,\varepsilon)$ to $T_{x_0}X$ so that $(T_{x_0}X, g^{(x_0)})$ is a manifold of bounded geometry, $L_0$ and $E_0$ have bounded geometry and $V^{{(x_0)}}\in C^\infty_b(T_{x_0}X,\operatorname{End}(E_0))$ be a self-adjoint endomorphism of $E_0$.  For instance, we can take a smooth even function $\rho : \mathbb R\to [0,1]$ supported in $(-r_X,r_X)$ such that $\rho(v)=1$ if $|v|<\varepsilon$ and introduce the map $\varphi : T_{x_0}X\to T_{x_0}X$ defined by $\varphi (Z)=\rho(|Z|_{g^{(x_0)}})Z$. Then we introduce the Riemannian metric $g^{(x_0)}_Z=g_{\varphi(Z)}, Z\in T_{x_0}X$, a Hermitian connection $\nabla^{L_0}$ on $(L_0,h^{L_0})$ by 
\[
\nabla^{L_0}_u=\nabla^L_{d\varphi(Z)(u)}. \quad Z\in T_{x_0}X, \quad u\in T_Z(T_{x_0}X), 
\]
where we use the canonical isomorphism $T_{x_0}X\cong T_Z(T_{x_0}X)$, and a Hermitian connection $\nabla^{E_0}$ on $(E_0,h^{E_0})$ by $\nabla^{E_0}=\varphi^*\nabla^E$. Finally, we set $V^{(x_0)}=\varphi^*V$. 

Let $\Delta^{L_0^p\otimes E_0}$ be the associated Bochner Laplacian acting on $C^\infty(T_{x_0}X,L_0^p\otimes E_0)$.  Introduce the operator $H^{(x_0)}_p$ acting on $C^\infty(T_{x_0}X,L_0^p\otimes E_0)$ by
\[
H^{(x_0)}_p=\frac 1p \Delta^{L_0^p\otimes E_0}+ V^{(x_0)}.
\]
It is clear that, for any $u \in {C}^\infty_c(B^{T_{x_0}X}(0,\varepsilon))$, we have
\begin{equation}\label{e:Hp=HX0p}
H_pu(Z)=H^{(x_0)}_pu(Z).
\end{equation}

Let $dv^{(x_0)}$ be the Riemannian volume form of $(T_{x_0}X, g^{(x_0)})$ and
\begin{equation}\label{e:4.6}
	K_{\varphi(H^{(x_0)}_p)}\in {C}^{\infty}(T_{x_0}X\times T_{x_0}X,
\pi_1^*(L_0^p\otimes E_0)\otimes \pi_2^*(L^p_0\otimes E_0)^*)
\end{equation}
the Schwartz kernel of the operator $\varphi(H^{(x_0)}_p)$ with respect to the volume form $dv^{(x_0)}$. Recall that the Schwartz kernel $K_{\varphi(H_p)}$ induces a smooth section $K_{\varphi(H_p),x_0}(Z,Z^\prime)$ of the vector bundle $\pi^*(\operatorname{End}(E))$ on $TX\times_X TX$ defined for all $x_0\in X$ and $Z,Z^\prime\in T_{x_0}X$ with $|Z|, |Z^\prime|<r_X$. We show that the kernels $K_{\varphi(H_p),x_0}(Z,Z^\prime)$ and $K_{\varphi(H^{(x_0)}_p)}(Z,Z^\prime)$ are asymptotically close on $B^{T_{x_0}X}(0,\varepsilon)$ in the $\mathcal C^\infty$-topology, as $p\to \infty$. 

\begin{prop}\label{p:Pq-difference}
For any $\varepsilon_1\in (0,\varepsilon)$, $k\in \mathbb N$ and $N\in \NN$, there exists $C>0$ such that 
 \[
|K_{\varphi(H_p),x_0}(Z,Z^\prime)-K_{\varphi(H^{(x_0)}_p)}(Z,Z^\prime)|_{\mathcal C^k}\leq Cp^{-N}
\]
for any $p\in\NN$, $x_0\in X$  and $Z,Z^\prime\in B^{T_{x_0}X}(0,\varepsilon_1)$.
 \end{prop}

\begin{proof}
As in \cite[Proposition 1.3]{ma-ma08}, this follows from \eqref{e:Hp=HX0p} and the finite propagation speed property of solutions of hyperbolic equations. Let us briefly recall the argument. Without loss of generality, we can assume that $H_p$ and $H^{(x_0)}_p$ are positive operators. Then we can write 
\[
\varphi(H_p)=\int_{-\infty}^{+\infty}\hat{\varphi}(t)e^{it\sqrt{H_p}}dt, \quad \varphi(H^{(x_0)}_p)=\int_{-\infty}^{+\infty}\hat{\varphi}(t)e^{it\sqrt{H^{(x_0)}_p}}dt
\]
By the finite propagation speed property of solutions of hyperbolic equations, 
the Schwartz kernel of $e^{it\sqrt{H_p}}$ is supported in the set 
\[
\{(x,x^\prime)\in X\times X : d_X(x,x^\prime)<\frac{t}{\sqrt{p}}\},
\]
a similar statement holds for $e^{it\sqrt{H^{(x_0)}_p}}$. Moreover, for any $u \in {C}^\infty_c(X, L^p\otimes E)$ the restriction of $e^{it\sqrt{H_p}}u$ to $B^{X}(x_0,\varepsilon-\frac{t}{\sqrt{p}})$ depends only on the restriction of $H_p$ and $u$ to $B^{X}(x_0,\varepsilon)$.  

Therefore, by \eqref{e:Hp=HX0p}, for any $\varepsilon_1\in (0,\varepsilon)$, $u \in {C}^\infty_c(B^{X}(x_0,\varepsilon_1), L^p\otimes E)\cong {C}^\infty_c(B^{X_0}(0,\varepsilon_1), L^p\otimes E)$ and $\frac{t}{\sqrt{p}}<\varepsilon-\varepsilon_1$, we have
\begin{equation}\label{e:Hp=HX0p-exp}
e^{it\sqrt{H_p}}u(Z)=e^{it\sqrt{H^{(x_0)}_p}}_pu(Z).
\end{equation}

Take a cut-off function $\chi\in C^\infty_c(\mathbb R)$ such that $0\leq \chi(t)\leq 1$ for all $t\in \mathbb R$, $\chi(t)\equiv 1$ in a neighborhood of $0$ and $\chi(t)= 0$ for $|t|>\varepsilon-\varepsilon_1$. Write
\begin{align*}
\varphi(H_p)=&\int_{-\infty}^{+\infty}\chi(\frac{t}{\sqrt{p}})\hat{\varphi}(t)e^{it\sqrt{H_p}}dt+\int_{-\infty}^{+\infty}(1-\chi(\frac{t}{\sqrt{p}}))\hat{\varphi}(t)e^{it\sqrt{H_p}}dt\\
=&\varphi_{1,p}(H_p)+\varphi_{2,p}(H_p),
\end{align*}
similarly for $\varphi(H^{X_0}_p)$. By \eqref{e:Hp=HX0p-exp}, for any $u \in {C}^\infty_c(B^{X}(x_0,\varepsilon_1), L^p\otimes E)\cong {C}^\infty_c(B^{X_0}(0,\varepsilon_1), L^p\otimes E)$, we have 
\[
\varphi_{1,p}(H_p)u=\varphi_{1,p}(H^{(x_0)}_p)u,
\]
which implies that 
\[
K_{\varphi_{1,p}(H_p),x_0}(Z,Z^\prime)=K_{\varphi_{1,p}(H^{(x_0)}_p)}(Z,Z^\prime), \quad Z,Z^\prime \in B^{T_{x_0}X}(0,\varepsilon_1). 
\]
On the other hand, using the fact that $\varphi_{2,p}=\mathcal O(p^{-\infty})$ in Schwartz class, it is easy to see that for any $k\in \mathbb N$ and $N\in \NN$, there exists $C>0$ such that 
 \[
|K_{\varphi_{2,p}(H_p),x_0}(Z,Z^\prime)|_{\mathcal C^k}+|K_{\varphi_{2,p}(H^{(x_0)}_p)}(Z,Z^\prime)|_{\mathcal C^k}\leq Cp^{-N}
\]
for any $p\in\NN$, $x_0\in X$  and $Z,Z^\prime\in B^{T_{x_0}X}(0,\varepsilon_1)$.
\end{proof}

Proposition~\ref{p:Pq-difference} reduces our considerations to the $C^\infty_b$-bounded family $H^{(x_0)}_{p}$ of second order elliptic differential operators acting on $C^\infty(T_{x_0}X, L_0^p\otimes E_0)\cong C^\infty(T_{x_0}X,E_{x_0})$ (parametrized by $x_0\in X$). Note that, once we moved to the Euclidean space, we can consider $p$ to be a real-valued parameter.

 \section{Rescaling and formal expansions}\label{scale}
We use the rescaling introduced in \cite[Section 1.2]{ma-ma08}. 

Denote $t=\frac{1}{\sqrt{p}}$. For $s\in C^\infty(T_{x_0}X, E_{x_0})$, set
\[
S_ts(Z)=s(Z/t), \quad Z\in T_{x_0}X.
\]
Consider the rescaling of the operator $H_p^{(x_0)}$ defined by 
\begin{equation}\label{scaling}
\mathcal H_t=S^{-1}_t\kappa^{\frac 12}H_p^{(x_0)}\kappa^{-\frac 12}S_t,
\end{equation}
where $\kappa=\kappa_{x_0}$ is defined in \eqref{e:def-kappa}.

By construction, it is a self-adjoint operator in $L^2(T_{x_0}X, E_{x_0})$, and its spectrum coincides with the spectrum of $H_p^{(x_0)}$.

From now on, we choose an orthonormal basis $\mathbf e=\{e_j, j=1,2,\ldots,d\}$ of $T_{x_0}X$. It gives rise to an isomorphism $T_{x_0}X\cong \mathbb R^{d}$. We will consider each $e_j$ as a constant vector field on $T_{x_0}X$ and use the same notation $e_j$ for the corresponding vector field on $\mathbb R^d$.  So $\{e_j, j=1,2,\ldots,d\}$ is the standard basis of $\mathbb R^d$.

We get a family of second order differential operators on $C^\infty(\RR^d, E_{x_0})$, which depend on the orthonormal basis $\mathbf e$ and will be also denoted by $\mathcal H_t$.  For simplicity of notation, we will skip $\mathbf e$. The operator $\mathcal H_t$ is given by the formula
\begin{equation}\label{e:Ht}  
\mathcal H_t=-\sum_{j,k=1}^{d} g^{jk}(tZ)\left[\nabla_{t,e_j}\nabla_{t,e_k}- t\sum_{\ell=1}^{d}\Gamma^{\ell}_{jk}(tZ)\nabla_{t,e_\ell}\right]+V(tZ),
\end{equation}
where the functions $\Gamma^{\ell}_{jk}$ are defined by the equality 
$\nabla_{e_j}e_k=\sum_{\ell=1}^d \Gamma^{\ell}_{jk}e_\ell$ and the rescaled connection $\nabla_t$ is given by
\begin{equation}\label{e:nablat}
\nabla_t=tS^{-1}_t\kappa^{\frac 12}\nabla^{L_0^p}\kappa^{-\frac 12}S_t.
\end{equation}

Let $\Gamma^{L_0}=\sum_{j=1}^{d}\Gamma_j^{L_0}(Z)dZ_j$, $\Gamma_j^{L_0}\in C^\infty(\mathbb R^{d})$, and $\Gamma^{E_0}=\sum_{j=1}^{d}\Gamma_j^{E_0}(Z)dZ_j$, $\Gamma_j^{E_0}\in C^\infty(\mathbb R^{d}, \operatorname{End}(E_{x_0}))$, be the connection forms of $\nabla^{L_0}$ and $\nabla^{E_0}$, respectively. Then 
\begin{equation}\label{E:nablatej}
\begin{split}
\nabla_{t,e_j}&=\kappa^{\frac 12}(tZ)\left(\nabla_{e_j}+\frac{1}{t}\Gamma^{L_0}_j(tZ)+t\Gamma^{E_0}_j(tZ)\right)\kappa^{-\frac 12}(tZ)\\
&=\nabla_{e_i}+\frac{1}{t}\Gamma^{L_0}_j(tZ)+t\Gamma^{E_0}_j(tZ)-t\left(\kappa^{-1}(e_j\kappa)\right)(tZ).
\end{split}
\end{equation}

Recall that 
\[
\Gamma^{L_0}_j(Z)=\frac{1}{2}R^L_{x_0}(\mathcal R,e_j)+O(|Z|^2). 
\]
This shows that the operators $\nabla_{t,e_i}$ depend smoothly on $t$ up to $t=0$, and the limit $\nabla_{0,e_j}$ of the connection $\nabla_{t,e_j}$ as $t\to 0$ coincides with the connection $\nabla^{(x_0)}$ given by \eqref{e:nablaL0}:
\begin{equation}\label{E:nabla0ej}
\nabla_{0,e_j}=\nabla_{e_j}+\frac{1}{2}R^L_{x_0}(\mathcal R, e_j)=\nabla^{(x_0)}_{e_j}. 
\end{equation}
 
Now we expand the coefficients of the operator $\mathcal H_t$ in Taylor series in $t$. For any $m\in \NN$, we get
\begin{equation}\label{e:Ht-formal}
\mathcal H_t=\mathcal H^{(0)}+\sum_{j=1}^m \mathcal H^{(j)}t^j+\mathcal O(t^{m+1}), 
\end{equation}
where there exists $m^\prime\in \NN$ so that for every $k\in\NN$ and $t\in [0,1]$ the derivatives up to order $k$ of the coefficients of the operator $\mathcal O(t^{m+1})$ are bounded by $Ct^{m+1}(1+|Z|)^{m^\prime}$. 

The leading term $\mathcal H^{(0)}$ coincides with the operator $\mathcal H^{(x_0)}$ given by \eqref{e:DeltaL0p}:
\[  
\mathcal H^{(0)}=-\sum_{j=1}^{d} (\nabla_{0,e_j})^2+V^{(x_0)}(0)=\mathcal H^{(x_0)}. 
\]
By \cite[Theorem 1.4]{ma-ma08}, the next terms $\mathcal H^{(j)}, j\geq 1,$ have the form
\begin{equation}\label{e:Hj}
\mathcal H^{(j)}=\sum_{k,\ell=1}^{d} a_{k\ell,j}\frac{\partial^2}{\partial Z_k\partial Z_\ell}+\sum_{k=1}^{d} b_{k,j}\frac{\partial}{\partial Z_k}+c_{j},
\end{equation}
where $a_{k\ell,j}$ is a homogeneous polynomial in $Z$ of degree $j$,  $b_{kj}$ is a polynomial in $Z$ of degree $\leq j+1$ (of the same parity with $j-1$) and $c_{j}$ is a polynomial in $Z$ of degree $\leq j+2$ (of the same parity with $j$). More precisely, for the operator $H_p=\frac{1}{p}\Delta_p$, the operator $\mathcal H^{(j)}$ coincides with the operator $\mathcal O_j$ introduced in that theorem. In the general case, we have 
\[
\mathcal H^{(j)}=\mathcal O_j+\sum_{|\alpha|=j}(\partial^\alpha(V+\tau))_{x_0}\frac{Z^\alpha}{\alpha!},\quad j=1,2,\ldots,
\]
In \cite[Theorem 1.4]{ma-ma08}, explicit formulas are given for $\mathcal O_1$ and $\mathcal O_2$. We refer the reader to \cite{ma-ma08,ma-ma:book} for more details. 

\section{Norm estimates}\label{norm}
In this section, we will establish norm estimates for the operators $\varphi(\mathcal H_{t})$ and its derivatives of any order with respect to $t$. We start with the resolvent of $\mathcal H_t$.

For $t>0$, set 
\[
\|s\|^2_{t,0}=\|s\|^2_{0}=\int_{\mathbb R^{d}}|s(Z)|^2dv_{X,x_0}(Z), \quad s\in C^\infty_c(\mathbb R^{d},E_{x_0}),
\]
and, for any $m\in \mathbb N$ and $t>0$,
\[
\|s\|^2_{t,m}=\sum_{\ell=0}^m\sum_{j_1,\ldots,j_\ell=1}^{d}\|\nabla_{t,e_{j_1}}\cdots \nabla_{t,e_{j_\ell}}s\|^2_{t,0},\quad s\in C^\infty_c(\mathbb R^{d},E_{x_0}), 
\]
where $\nabla_t$ is the rescaled connection defined by \eqref{e:nablat}. 

Let $\langle\cdot,\cdot\rangle_{t,m}$ denote the inner product on $C^\infty_c(\mathbb R^{d},E_{x_0})$ corresponding to $\|\cdot\|^2_{t,m}$. Let $H^m_t$ be the Sobolev space of order $m$ with norm $\|\cdot\|_{t,m}$. For any integer $m<0$, we define the Sobolev space $H^{m}_t$ by duality. For any bounded linear operator $A: H^m_t\to H^{m^\prime}_t$ with $m,m^\prime\in \mathbb Z$, we denote by $\|A\|^{m,m^\prime}_t$ its norm with respect to $\|\cdot\|_{t,m}$ and $\|\cdot\|_{t,m^\prime}$.  

For a fixed $t>0$, the norm $\|\cdot\|_{t,m}$ is equivalent to the standard Sobolev norm given for $m\in \mathbb N$, by 
\[
\|s\|^2_{H^m(\mathbb R^{d},E_{x_0})}=\sum_{\ell=0}^m\sum_{j_1,\ldots,j_\ell=1}^{d}\|\nabla_{e_{j_1}}\cdots \nabla_{e_{j_\ell}}s\|^2_{0},\quad s\in C^\infty_c(\mathbb R^{d},E_{x_0}).
\]
Therefore, the space $H^{m}_t$ coincides with the usual Sobolev space
$H^m(\mathbb R^{d},E_{x_0})$ as a topological vector space. But this norm equivalence is not uniform as $t\to 0$. In order to have control of the Sobolev norms by the norms $\|\cdot\|_{t,m}$, uniform up to $t=0$, we introduce some weighted Sobolev norms with power weights as in \cite{ma-ma08}.

For $\alpha\in \ZZ^d_+$, we will use the standard notation
\[
Z^\alpha=Z_1^{\alpha_1}Z_2^{\alpha_2}\ldots Z_d^{\alpha_d}, \quad Z\in \RR^d. 
\]
For any $t>0$, $m\in \mathbb Z$, and $M\in \ZZ_+$, we set
\[
\|s\|_{t,m,M}:= \sum_{|\alpha|\leq M}\left\|Z^{\alpha} s\right\|_{t,m}, \quad s\in C^\infty_c(\mathbb R^{d},E_{x_0}).
\]
By \eqref{E:nablatej}, any operator of the form $\nabla_{e_{j_1}}\cdots \nabla_{e_{j_\ell}}$ can be written as 
\begin{equation}\label{e:nabla-nablat}
\nabla_{e_{j_1}}\cdots \nabla_{e_{j_\ell}}=\sum_{k=0}^\ell \sum_{i_1,\ldots,i_k} A_{t,i_1,\ldots,i_k} \nabla_{t, e_{i_1}}\cdots \nabla_{t,e_{i_k}}, \quad t>0,
\end{equation}
where $A_{t,i_1,\ldots,i_k}\in C^\infty(\mathbb R^{d},\operatorname{End}(E_{x_0}))$ satisfies the following condition: for any $\beta\in \mathbb Z_+^{d}$, there exists $C_\beta>0$ such that
\[
|\nabla^{\beta_1}_{e_1}\ldots \nabla^{\beta_{d}}_{e_{d}} A_{t,i_1,\ldots,i_k} (Z)| <C_\beta (1+|Z|)^\ell, \quad Z\in \mathbb R^{d},\quad t\in (0,1].
\]
By \eqref{e:nabla-nablat}, it follows that, for any $m\in \NN$, there exists $C>0$ such that 
\begin{equation}\label{e:Hm-Hmm}
\|s\|_{H^m(\mathbb R^{d},E_{x_0})} \leq C \|s\|_{t,m,m},\quad t\in (0,1], \quad s\in C^\infty(\mathbb R^{d},E_{x_0}).
\end{equation}

Moreover, we will use an approach to the proof of the full-diagonal expansions developed in \cite{Kor18}. Therefore, we introduce an additional family of weight functions parameterized by a point $W\in \mathbb R^d$. Unlike \cite{Kor18} where these weight functions are exponential, we will consider power weights. 

For any $t>0$, $m\in \mathbb Z$, $M,N\in \ZZ_+$ and $W\in \RR^d$, we set
\[
\|s\|_{t,m,M,N,W}= \sum_{|\alpha|\leq M,|\beta|\leq N}\left\|Z^{\alpha}(Z-W)^\beta s\right\|_{t,m}, \quad s\in C^\infty_c(\mathbb R^{d},E_{x_0}).
\]
It is clear that
\[
\|s\|_{t,m,M,0,W}=\|s\|_{t,m,M}.
\]

Since $\mathcal H_t$ is self-adjoint and uniformly elliptic, for any $t>0$ and $\lambda\in \CC\setminus \RR$, the operator $\lambda-\mathcal H_{t}$ is invertible in $L^2(\mathbb R^{d},E_{x_0})$, and the inverse operator $\left(\lambda-\mathcal H_{t}\right)^{-1}$ maps $H^m(\mathbb R^{d},E_{x_0})$ to $H^{m+2}(\mathbb R^{d},E_{x_0})$.

\begin{thm}\label{Thm1.9}
For any $t\in (0,1]$, $m\in \mathbb N$, $M,N\in \ZZ_+$, $W\in \RR^d$ and $\lambda=\mu+i\nu\in \CC\setminus \RR$, we have 
\begin{equation}\label{e:mm+2t}
\left\|(\lambda-\mathcal H_{t})^{-1}s\right\|_{t,m+2,M,N,W}\leq C|\nu|^{-M-N-1}\left\|s\right\|_{t,m,M,N,W}, 
\end{equation}
where $s\in C^\infty_c(\mathbb R^{d},E_{x_0})$ and $C=C_{m,M,N}>0$ is independent of $t\in (0,1]$, $\lambda\in  \CC\setminus \RR$, $W\in \RR^d$ and $x_0\in X$.
\end{thm}

\begin{proof}
For any $t>0$, we have an isometric isomorphism
\[
U_t =t^nS^{-1}_t\kappa^{\frac 12} : L^2(T_{x_0}X,L_0^p\otimes E_0)\to H^0_t.
\]
We have
\[
\nabla_t=U_t\left(\frac{1}{\sqrt{p}}\nabla^{L^p_0\otimes E_0}\right)U^{-1}_{t}, \quad \mathcal H_t=U_tH^{(x_0)}_pU^{-1}_{t}.
\]
Therefore, $U_t$ extends to an isometric isomorphism
\[
U_t : H^m(T_{x_0}X,L^p_0\otimes E_0)\to H^m_t,\quad m\in \ZZ.
\] 
Therefore, the estimate \eqref{e:mm+2t} for $M=N=0$ follows from a similar estimate for the operator $H^{(x_0)}_p$ proved in \cite[Theorem~4]{higherLL}. We only note that the proof of the latter theorem does not use the non-degeneracy of the curvature $R^L$.  

Now we proceed by induction. We assume that \eqref{e:mm+2t} holds for all $M,N\in \ZZ_+$ with $M+N=k$ and for any $W\in \RR^d$. Take $\alpha, \beta\in \ZZ^d_+$ with $|\alpha|=M, |\beta|=N$, $M+N=k+1$. Then we write
\begin{multline}\label{e:est1}
\left\|Z^{\alpha}(Z-W)^\beta (\lambda-\mathcal H_{t})^{-1} s\right\|_{t,m}\\ 
\leq \left\|(\lambda-\mathcal H_{t})^{-1}Z^{\alpha}(Z-W)^\beta s\right\|_{t,m} +\left\|[Z^{\alpha}(Z-W)^\beta,(\lambda-\mathcal H_{t})^{-1}] s\right\|_{t,m}
\end{multline}
and
\begin{equation}\label{e:comm-res}
[Z^{\alpha}(Z-W)^\beta,(\lambda-\mathcal H_{t})^{-1}]=(\lambda-\mathcal H_{t})^{-1}[Z^{\alpha}(Z-W)^\beta, \mathcal H_{t}](\lambda-\mathcal H_{t})^{-1}.
\end{equation}
By \eqref{e:Ht}, the commutator $[Z^{\alpha}(Z-W)^\beta, \mathcal H_{t}]$ is a first order differential operator given by
%\[
%[Z_\ell, \mathcal H_{t}]= -\sum_{j,k=1}^{d} g^{jk}(tZ)\left[\delta_{j\ell}\nabla_{t;e_k}+\delta_{k\ell}\nabla_{t;e_j}- t\Gamma^{\ell}_{jk}(tZ)\right].
%\]
\begin{multline*}
[Z^{\alpha}(Z-W)^\beta,\mathcal H_t] \\ =-\sum_{j,k=1}^{d} g^{jk}(tZ)\Big[[Z^{\alpha}(Z-W)^\beta,\nabla_{t,e_j}]\nabla_{t,e_k}+\nabla_{t,e_j}[Z^{\alpha}(Z-W)^\beta,\nabla_{t,e_k}]\\ - t\sum_{\ell=1}^{d}\Gamma^{\ell}_{jk}(tZ)[Z^{\alpha}(Z-W)^\beta,\nabla_{t,e_\ell}]\Big].
\end{multline*}

Using the commutation relation
\begin{equation}\label{e:Z-nabla-com}
[Z_\ell,\nabla_{t;e_j}]=\delta_{j\ell},
\end{equation}
we get 
\begin{equation}\label{e:Z-nabla1-com}
[Z^{\alpha}(Z-W)^\beta,\nabla_{t;e_j}]=\alpha_\ell Z^{\alpha-\delta_\ell}(Z-W)^\beta+\beta_\ell Z^{\alpha}(Z-W)^{\beta-\delta_\ell},
\end{equation}
where $\delta_\ell\in \ZZ^d_+$ is given by $(\delta_{\ell})_j=\delta_{\ell j}, j=1,2,\ldots,d$. 
Therefore, for any $t\in (0,1]$, $m\in \mathbb N$, $M,N\in \ZZ_+$ and $W\in \RR^d$, we have
\begin{equation}\label{e:est-comm}
\|[Z^{\alpha}(Z-W)^\beta,\nabla_{t;e_j}]s\|_{t,m}\leq C_1 (\|s\|_{t,m,M-1,N,W} +\|s\|_{t,m,M,N-1,W})
\end{equation}
and 
\begin{equation}\label{e:est-nablat}
\|\nabla_{t;e_j} s\|_{t,m,M,N,W}\leq C_2\|s\|_{t,m+1,M,N,W}, \quad s\in C^\infty_c(\mathbb R^{d},E_{x_0}),
\end{equation}
%and 
%\[
%\|\mathcal H_t s\|_{t,m,M,N,W}\leq C_2\|s\|_{t,m+2,M+2,N,W}, \quad s\in C^\infty_c(\mathbb R^{d},E_{x_0}),
%\]
where $C_1, C_2>0$ are independent of $t\in (0,1]$, $W\in \RR^d$ and $x_0\in X$. Here we set $\|s\|_{t,m^\prime,-1,N,W}=\|s\|_{t,m^\prime,M,-1,W}=0$.

Using \eqref{e:est-comm} and \eqref{e:est-nablat}, we conclude that, for any $t\in (0,1]$, $m^\prime\in \mathbb N$, $M,N\in \ZZ_+$ and $W\in \RR^d$,
\begin{multline}\label{e:est-comm-res}
\|[Z^{\alpha}(Z-W)^\beta,\mathcal H_t]s\|_{t,m^\prime}\\ \leq C (\|s\|_{t,m^\prime+1,M-1,N,W} +\|s\|_{t,m^\prime+1,M,N-1,W}).
\end{multline}

Using induction hypothesis, \eqref{e:comm-res} and \eqref{e:est-comm-res}, we proceed in \eqref{e:est1} as follows
\begin{multline*}
\left\|Z^{\alpha}(Z-W)^\beta (\lambda-\mathcal H_{t})^{-1} s\right\|_{t,m}\\ 
\begin{aligned}
\leq & C |\nu|^{-1}  \left(\left\|s\right\|_{t,m-2,M,N,W}+\left\|[Z^{\alpha}(Z-W)^\beta, \mathcal H_{t}] (\lambda-\mathcal H_{t})^{-1} s\right\|_{t,m-2}\right)\\
\leq & C |\nu|^{-1} \Big(\left\|s\right\|_{t,m-2,M,N,W}+\left\|(\lambda-\mathcal H_{t})^{-1} s\right\|_{t,m-1,M-1,N,W}\\ & +\left\|(\lambda-\mathcal H_{t})^{-1} s\right\|_{t,m-1,M,N-1,W}\Big)\\
\leq & C |\nu|^{-M-N-1}\|s\|_{t,m-2,M,N,W},
\end{aligned}
\end{multline*}
that completes the proof of \eqref{e:mm+2t}. 
\end{proof}

By the Helffer-Sjostrand formula \cite{HS-LNP345}, we have  
\begin{equation}\label{e:HS}
\varphi(\mathcal H_{t})=-\frac{1}{\pi }
\int_\CC \frac{\partial \tilde{\varphi}}{\partial \bar \lambda}(\lambda)(\lambda-\mathcal H_{t})^{-1}d\mu d\nu,
\end{equation}
where $\tilde{\varphi}\in C^\infty_c(\CC)$ is an almost-analytic extension of $\varphi$ satisfying 
\[
\frac{\partial \tilde{\varphi}}{\partial \bar \lambda}(\lambda)=O(|\nu|^\ell),\quad \lambda=\mu+i\nu, \quad \nu\to 0,
\]
for any $\ell\in \NN$.

Let us consider the function $\psi(\lambda)=\varphi(\lambda)(a-\lambda)^K$ with any $K\in \NN$ and $a>0$. Then its almost-analytic extension $\tilde{\psi}$ can be taken to be $\tilde{\psi}(\lambda)=\tilde{\varphi}(\lambda)(a-\lambda)^K$. If we apply the formula \eqref{e:HS} to $\psi$, then we get 
\begin{equation}\label{e:HS1}
\varphi(\mathcal H_{t})=-\frac{1}{\pi }
\int_\CC \frac{\partial \tilde{\varphi}}{\partial \bar \lambda}(\lambda) (a-\lambda)^K (\lambda-\mathcal H_{t})^{-1}(a-\mathcal H_{t})^{-K}d\mu d\nu.
\end{equation}

\begin{prop}
For any $t\in (0,1]$ and $m, m^\prime\in \ZZ$, the operator $\varphi(\mathcal H_{t})$ extends to a bounded operator from $H^m_t$  to $H^{m^\prime}_t$ with the following norm estimate for any $M,N\in \mathbb Z_+$ and $W\in \RR^d$:
\begin{equation}\label{e:varphi-Ht}
\left\|\varphi(\mathcal H_{t})s\right\|_{t,m^\prime,M,N,W}\leq C\left\|s\right\|_{t,m,M,N,W}, \quad s\in C^\infty_c(\mathbb R^{d},E_{x_0}),
\end{equation}
where $C=C_{M,N,m.m^\prime}>0$ is independent of $t\in (0,1]$, $W\in \RR^d$ and $x_0\in X$.  
\end{prop}

\begin{proof}
 By Theorem \ref{Thm1.9}, it follows that, for any $m^\prime, K\in \NN$ and $M,N\in \mathbb Z_+$, there exists $C>0$ such that, for all $t\in (0,1]$, and $\lambda\in\CC\setminus \RR$
\begin{multline*}
\left\|(\lambda - \mathcal H_{t})^{-1} (a - \mathcal H_{t})^{-K}s\right\|_{t,m^\prime,M,N,W}\\ \leq C|\nu|^{-M-N-1}\left\|s\right\|_{t,m^\prime-2(K+1),M,N,W}, \quad s\in C^\infty_c(\mathbb R^{d},E_{x_0}). 
\end{multline*}
By the above estimates, the desired statement follows immediately from the formula \eqref{e:HS1} with  appropriate $K$.
\end{proof}

\begin{thm}\label{est-rem}
For any $r\geq 0$, $m, m^\prime\in \ZZ$, and $M,N\in \mathbb Z_+$, there exists $C>0$ such that, for any $t\in (0,t_0]$ and $W\in \RR^d$, 
\[
\left\|\frac{\partial^r}{\partial t^r}\varphi(\mathcal H_{t}) s\right\|_{t,m,M,N,W}\leq C\|s\|_{t,m^\prime,M+2r,N,W}, \quad s\in C^\infty_c(\mathbb R^{d},E_{x_0}).
\] 
\end{thm}

\begin{proof}
We proceed as in the proof of \cite[Theorem 1.10]{ma-ma08}. By \eqref{e:HS1}, we have 
\begin{equation}\label{e:diff-phi}
\frac{\partial^r}{\partial t^r}\varphi(\mathcal H_{t})=-\frac{1}{\pi }
\int_\CC \frac{\partial \tilde{\varphi}}{\partial \bar \lambda}(a-\lambda)^K \frac{\partial^r}{\partial t^r}\left[(\lambda-\mathcal H_{t})^{-1}(a-\mathcal H_{t})^{-K}\right] d\mu d\nu,
\end{equation}

We set 
\[
I_{r}=\left\{ {\mathbf r}=(r_1,\ldots,r_j) : \sum_{i=1}^jr_i=r, r_i\in \mathbb N \right\}. 
\]
Then we write
\begin{multline}\label{diff}
\frac{\partial^r}{\partial t^r}(\lambda-\mathcal H_{t})^{-1}\\ =\sum_{{\mathbf r}\in I_{r}} \frac{r!}{{\mathbf r}!} (\lambda - \mathcal H_{t})^{-1} \frac{\partial^{r_1}\mathcal H_{t}}{\partial t^{r_1}}(\lambda - \mathcal H_{t})^{-1}\cdots  \frac{\partial^{r_j}\mathcal H_{t}}{\partial t^{r_j}}(\lambda - \mathcal H_{t})^{-1}.
\end{multline}
 
By \eqref{E:nablatej}, we obtain that, for $r>0$, $\frac{\partial^r}{\partial t^r}\nabla_{t,e_j}$ is a function of the form $f_r(tZ)Z^\beta$ with $|\beta|\leq r+1$ and $f_r\in C^\infty_b(\RR^{d})$. We also have that, for $r>0$, $\frac{\partial^r}{\partial t^r}(g(tZ))$ is a function of the form $g_r(tZ)Z^\beta$ with $|\beta|\leq r$ and $g_r\in C^\infty_b(\RR^{d})$. Using these facts, \eqref{e:Ht} and \eqref{e:Z-nabla-com}, we infer that for any $r$ the operator $\frac{\partial^{r}\mathcal H_{t}}{\partial t^{r}}$ is a second order differential operator of the form
\begin{multline}\label{e:partial-rLt}
\frac{\partial^{r}\mathcal H_{t}}{\partial t^{r}}= \sum_{j,k=1}^{d}\sum_{|\beta|\leq r} A^{jk}_\beta (tZ)Z^\beta \nabla_{t;e_j}\nabla_{t;e_k}\\ +\sum_{\ell=1}^{d}\sum_{|\beta|\leq r+1} (B^{\ell}_{\beta}(tZ)+t\hat B^{\ell}_{\beta}(tZ))Z^\beta\nabla_{t,e_\ell}+\sum_{|\beta|\leq r+2} (C_\beta(tZ)+t\hat C_\beta(tZ)) Z^\beta.
\end{multline}
It follows that, for any $m\in \NN$ and $M,N\in \mathbb Z_+$, there exists $C>0$ such that, for all $t\in (0,1]$ and $W\in \RR^d$,  
\begin{equation}\label{e:partial-rL}
\left\|\frac{\partial^{r}\mathcal H_{t}}{\partial t^{r}} s\right\|_{t,m,M,N,W}\leq C\left\|s\right\|_{t,m+2,M+r,N,W}, \quad s\in C^\infty_c(\mathbb R^{d},E_{x_0}). 
\end{equation}

Using \eqref{diff}, \eqref{e:mm+2t} and \eqref{e:partial-rL}, we obtain that for any $m\in \ZZ$ and $M,N\in \mathbb Z_+$, there exists $C>0$ such that for $\lambda\in \CC\setminus\RR$, $t\in (0,1]$ and $W\in \RR^d$, 
\begin{multline}\label{e:res-mm-prime}
\left\| \frac{\partial^r}{\partial t^r}\left[(\lambda-\mathcal H_{t})^{-1}(a-\mathcal H_{t})^{-K}\right] s\right\|_{t,m,M,N,W}\\ \leq C_{N,m,M}|\nu|^{-(M+r+1)}\left\|s\right\|_{t,m-2(K+1),M+2r,N,W}, \quad s\in C^\infty_c(\mathbb R^{d},E_{x_0}). 
\end{multline} 
Using \eqref{e:diff-phi} with appropriate $K$ and \eqref{e:res-mm-prime}, we complete the proof. 
\end{proof}

\section{Pointwise estimates}

In this section, we prove a general result, which allows us to derive pointwise estimates of the Schwartz kernels of smoothing operators from their mapping properties in Sobolev spaces.  
 
\begin{thm}\label{t:pointwise}
Assume that $\{K_t : C^\infty_c(\mathbb R^{d},E_{x_0}) \to C^\infty(\mathbb R^{d},E_{x_0}) : t\in (0,1]\}$ is a family of operators with smooth kernel $K_t(Z,Z^\prime)$ such that there exists $K>0$ such that, for any $m, m^\prime \in \NN$ and $M,N\in \mathbb Z_+$, there exists $C>0$ such that, for any $t\in (0,1]$ and $W\in \mathbb R^{d}$, 
\begin{equation}\label{e:est-Kt}
\|K_ts\|_{t,m,M,N,W}\leq C\|s\|_{t,-m^\prime ,M+K,N,W}, \quad s\in C^\infty_c(\mathbb R^{d},E_{x_0}).
\end{equation}
Then, for any $m\in \mathbb N$, there exists $M^\prime>0$ such that for any $N\in \NN$ there exists $C>0$ such that, for any $t\in (0,1]$ and $Z,Z^\prime\in \RR^{d}$ 
\[
\sup_{|\alpha|+|\alpha^\prime|\leq m}\Bigg|\frac{\partial^{|\alpha|+|\alpha^\prime|}}{\partial Z^\alpha\partial Z^{\prime\alpha^\prime}}K_{t}(Z,Z^\prime)\Bigg| 
\leq C(1+|Z|+|Z^\prime|)^{M^\prime}(1+|Z-Z^\prime|)^{-N}. 
\]
\end{thm}

\begin{proof}
The proof is based on the Sobolev embedding theorem.
Denote by $C^\infty_b(\mathbb R^{d},E_{x_0})$ the space of smooth functions on $\mathbb R^{d}$ with values in $E_{x_0}$ whose derivatives of any order are uniformly bounded in $\mathbb R^{d}$. So $a\in C^\infty_b(\mathbb R^{d},E_{x_0})$ if, for any $\alpha\in \mathbb Z_+^{d}$, we have  
\[
\sup_{Z\in \mathbb R^{d}}|\nabla^{\alpha_1}_{e_1}\ldots \nabla^{\alpha_{d}}_{e_{d}} a(Z)| <\infty.
\]
The Sobolev embedding theorem states that, for any $m\in \ZZ_+$, we have a continuous embedding 
\begin{equation}\label{e:Sobolev-emb}
H^{m+\lfloor d/2\rfloor+1}(\mathbb R^{d},E_{x_0})\hookrightarrow C^m_b(\mathbb R^{d},E_{x_0}),
\end{equation}
where $\lfloor x\rfloor$ denotes the largest integer less than or equal to $x$.

For any $v\in E_{x_0}$, consider the delta-function $\delta^v_Z\in \mathcal D^\prime (\mathbb R^{d},E_{x_0})$ at $Z\in \mathbb R^{d}$ defined by 
\[
\langle \delta^v_Z,s\rangle=\langle v, s(Z)\rangle_{h^{E_{x_0}}},\quad  s\in C^{\infty}_c(\mathbb R^{d},E_{x_0}).
\]  

\begin{lem}
For any $Z\in \mathbb R^{d}$, $v\in E_{x_0}$ and $t\in (0,1]$, the delta-function $\delta^v_{Z}$ belongs to $H^{-(\lfloor d/2\rfloor+1)}_t$. Moreover, there exists $C>0$ such that 
\begin{equation}\label{e:deltaZ}
\|\delta^v_Z\|_{t,-(\lfloor d/2\rfloor+1)}<C(1+|Z|)^{\lfloor d/2\rfloor+1}|v|, \quad Z\in \RR^{d}, v\in E_{x_0}, t\in (0,1].  
\end{equation}
\end{lem}

\begin{proof}
Take any function $\chi\in C^\infty_c(\RR^{d})$ supported in the unit ball $B(0,1)$ such that $\chi(0)=1$. For  $Z\in \RR^{d}$, define a function $\chi_Z\in C^\infty_s(\RR^{d})$ by $\chi_Z(W)=\chi(W-Z)$ for $W\in \RR^{d}$. By \eqref{e:Sobolev-emb} and \eqref{e:Hm-Hmm}, for any $s\in C^\infty_s(\RR^{d},E_{x_0})$, we have 
\[
|s(Z)|=|\chi_Zs(Z)|\leq C\|\chi_Zs\|_{H^{\lfloor d/2\rfloor+1}(\mathbb R^{d},E_{x_0})}\leq C_1\|\chi_Zs\|_{t,\lfloor d/2\rfloor+1,\lfloor d/2\rfloor+1}. 
\]
Since $\chi_Z$ is supported in $B(0,1+|Z|)$ and the multiplication operator by $\chi_Z$ defiines a bounded operator from $H^{\lfloor d/2\rfloor+1}_t$ to $H^{\lfloor d/2\rfloor+1}_t$ with the norm uniformly bounded on $t\in (0,1]$ and $Z\in \RR^{d}$, we proceed as follows: 
\[
|s(Z)|\leq C(1+|Z|)^{\lfloor d/2\rfloor+1}\|\chi_Zs\|_{t,\lfloor d/2\rfloor+1}\leq C_1(1+|Z|)^{\lfloor d/2\rfloor+1}\|s\|_{t,\lfloor d/2\rfloor+1}
\]
with the constant $C_1$, independent of $Z\in \RR^{d}$ and $t\in [0,1]$. By definition of the dual norm, we have
\[
\|\delta^v_Z\|_{t,-(\lfloor d/2\rfloor+1)}=\sup_{s \in C^\infty_s(\RR^{d},E_{x_0})} \frac{|\langle \delta^v_{Z},s\rangle|}{\|s\|_{t,\lfloor d/2\rfloor+1}},
\]
that immediately completes the proof.
\end{proof}

Recall that, for any operator $A: C^{\infty}_c(\mathbb R^{d},E_{x_0})\to C^{\infty}(\mathbb R^{d},E_{x_0})$ with smooth kernel $K_A\in C^{\infty}(\mathbb R^{d}\times \mathbb R^{d}, \operatorname{End}(E_{x_0}))$ given, for $u\in C^{\infty}_c(\mathbb R^{d},E_{x_0})$, by the formula
\[
Au(Z)=\int_{\mathbb R^{d}}K_A(Z,Z^\prime)u(Z^\prime)dZ^\prime, \quad Z\in \mathbb R^{d},
\] 
we have
\[
K_A(Z,Z^\prime)v=[A\delta^v_{Z^\prime}](Z), \quad Z,Z^\prime \in \RR^{d},\quad v\in E_{x_0}. 
\]
Moreover, using integration by parts, one can easily see that, for any $\alpha^\prime\in \ZZ_+^{d}$, the function $\frac{\partial^{|\alpha^\prime|}}{\partial Z^{\prime\alpha^\prime}}K_A(Z,Z^\prime)$
is the kernel of the operator $AD^\prime_{\alpha^\prime}$. where $D^\prime_{\alpha^\prime}=(-1)^{|\alpha^\prime|} \nabla_{e_1}^{\alpha^\prime_1}\nabla_{e_2}^{\alpha^\prime_2}\ldots \nabla_{e_{d}}^{\alpha^\prime_{d}}$. By \eqref{e:nabla-nablat}, it follows that, for any $m\in \ZZ$ and $M,N\in \ZZ_+$, there exists $C>0$ such that, $t\in (0,1]$ and $W\in \RR^d$, 
\begin{equation}\label{e:Da-est}
\|D^\prime_{\alpha^\prime}s\|_{t,m,M,N,W} \leq C \|s\|_{t,m+|\alpha^\prime|,M+|\alpha^\prime|,N,W},\quad s\in C^\infty(\mathbb R^{d}, E_{x_0}).
\end{equation}

Using these facts and \eqref{e:Sobolev-emb}, for any $m\in \ZZ_+$, $Z,Z^\prime \in \RR^{d}$ and $v\in E_{x_0}$, we proceed as follows:
\begin{multline*}
\sup_{|\alpha|+|\alpha^\prime|\leq m}\left|(1+|Z-Z^\prime|^2)^k \frac{\partial^{|\alpha|+|\alpha^\prime|}}{\partial Z^\alpha\partial Z^{\prime\alpha^\prime}} K_{t}(Z,Z^\prime)v\right|\\
\begin{aligned}
\leq & C_1\sup_{|\alpha^\prime|\leq m}\left\|(1+|W-Z^\prime|^2)^k\frac{\partial^{|\alpha^\prime|}}{\partial Z^{\prime\alpha^\prime}}K_{t}(W,Z^\prime)v\right\|_{C^m_b(\mathbb R^{d}_W,E_{x_0})}\\
= & C_1\sup_{|\alpha^\prime|\leq m}\left\|(1+|W-Z^\prime|^2)^k K_{t}D^\prime_{\alpha^\prime}\delta^v_{Z^\prime}\right\|_{C^m_b(\mathbb R^{d}_W,E_{x_0})}\\
\leq & C_2\sup_{|\alpha^\prime|\leq m}\left\|(1+|W-Z^\prime|^2)^k K_{t}D^\prime_{\alpha^\prime}\delta^v_{Z^\prime}\right\|_{H^{m+\lfloor d/2\rfloor+1}(\mathbb R^{d}_W,E_{x_0})}.
\end{aligned}
\end{multline*}
Using \eqref{e:est-Kt}, \eqref{e:Hm-Hmm}, \eqref{e:deltaZ} and \eqref{e:Da-est}, we get for $|\alpha^\prime|\leq m$
\begin{multline*}
\left\|(1+|W-Z^\prime|^2)^k K_{t}D^\prime_{\alpha^\prime}\delta^v_{Z^\prime}\right\|_{H^{m+\lfloor d/2\rfloor+1}(\mathbb R^{d}_W,E_{x_0})}\\
\begin{aligned}
 & \leq C_1 \left\|(1+|W-Z^\prime|^2)^k K_{t}D^\prime_{\alpha^\prime}\delta^v_{Z^\prime}\right\|_{t,m+\lfloor d/2\rfloor+1,m+\lfloor d/2\rfloor+1}\\ 
 & \leq C_1 \left\|K_{t}D^\prime_{\alpha^\prime}\delta^v_{Z^\prime}\right\|_{t,m+\lfloor d/2\rfloor+1,m+\lfloor d/2\rfloor+1,2k,Z^\prime}\\ 
 & \leq C_2\|D^\prime_{\alpha^\prime}\delta^v_{Z^\prime}\|_{t,-(m+\lfloor d/2\rfloor+1),m+\lfloor d/2\rfloor+K+1,2k,Z^\prime}\\ & \leq C_3\|\delta^v_{Z^\prime}\|_{t,-(\lfloor d/2\rfloor+1),2m+\lfloor d/2\rfloor+K+1,2k,Z^\prime}\\ & \leq C_4(1+|Z^\prime|)^{2m+\lfloor d/2\rfloor+K+1} \|\delta^v_{Z^\prime}\|_{t,-(\lfloor d/2\rfloor+1)}\\ & \leq C_5 (1+|Z^\prime|)^{M^\prime}|v|,
 \end{aligned}
\end{multline*}
where $M^\prime=2m+d+K+2$ that completes the proof. 
\end{proof}

Let $K_{\varphi(\mathcal H_{t})}(Z,Z^\prime)=K_{\varphi(\mathcal H_{t,x_0})}(Z,Z^\prime)$ be the smooth Schwartz kernel of the operator $\varphi(\mathcal H_{t})$ with respect to $dZ$. 
As an immediate consequence of Theorems \ref{est-rem} and \ref{t:pointwise}, we get the following result. 

\begin{thm}\label{est-rem-pointwise}
For any $r\geq 0$, $m\in \mathbb N$ and $N\in \NN$, there exist $C>0$ and $M>0$ such that for any $t\in (0,1]$ and $Z,Z^\prime\in \RR^{d}$, 
\[
\sup_{|\alpha|+|\alpha^\prime|\leq m}\Bigg|\frac{\partial^{|\alpha|+|\alpha^\prime|}}{\partial Z^\alpha\partial Z^{\prime\alpha^\prime}}\frac{\partial^r}{\partial t^r}K_{\varphi(\mathcal H_{t}) }(Z,Z^\prime)\Bigg| \leq C(1+|Z|+|Z^\prime|)^{M}(1+|Z-Z^\prime|)^{-N}. 
\]
\end{thm}

\section{The proof of asymptotic expansions} \label{asymp}
In this section, we prove asymptotic expansions of Theorem~\ref{t:main}. We start with the following analog of \cite[Theorem 1.11]{ma-ma08}. 

\begin{thm}\label{t:limitt-0} 
For any $Z,Z^\prime\in \RR^{d}$ and $r\geq 0$, there exists the limit 
\begin{equation}\label{e:defFr}
\lim_{t\to 0}\frac{\partial^r}{\partial t^r}K_{\varphi(\mathcal H_{t})}(Z,Z^\prime) =F_{r}(Z,Z^\prime),
\end{equation}
for some $F_{r}=F_{r,x_0}\in C^\infty(\mathbb R^{d}\times \mathbb R^{d},\operatorname{End}(E_{x_0}))$.
\end{thm}
 
\begin{proof}
The theorem follows from Theorem~\ref{est-rem-pointwise} and the following fact, which is an immediate consequence of the mean-value theorem: If $f$ is a differentiiable function on the interval $(0,1)$ with values in a Banach space $B$ such that $\sup_{t\in (0,1)} \|f^\prime(t)\|_B<\infty$. then there exists $\lim_{t\to 0}f(t)$. 
\end{proof}

By Theorem~\ref{est-rem-pointwise}, for any $r\geq 0$ and $m\in \mathbb N$ there exists $M>0$ such that for any $N\in \NN$, there exists $C>0$ such that for any  $Z,Z^\prime\in \RR^{d}$, 
\[
\sup_{|\alpha|+|\alpha^\prime|\leq m}\Bigg|\frac{\partial^{|\alpha|+|\alpha^\prime|}}{\partial Z^\alpha\partial Z^{\prime\alpha^\prime}}F_r(Z,Z^\prime)\Bigg| \leq C(1+|Z|+|Z^\prime|)^{M}(1+|Z-Z^\prime|)^{-N}. 
\]
Consider an operator $F_r : C^\infty_c(\mathbb R^{d},E_{x_0})\to C^\infty(\mathbb R^{d},E_{x_0})$ defined by the Schwartz kernel $F_r(Z,Z^\prime)$. By Theorem~\ref{est-rem}, for any $r\geq 0$, $m, m^\prime\in \ZZ$, and $N\in \mathbb Z_+$, there exists $C>0$ such that
\[
\|F_r s\|_{0,m,N}\leq C\|s\|_{0,m^\prime,N+2r}, \quad s\in C^\infty_c(\mathbb R^{d},E_{x_0}).
\]

 \begin{thm}\label{t:thm7.2}
For any $j,m,m^\prime\in \mathbb N$, there exists $M>0$ such that, for any $N\in \NN$,  there exists $C>0$ such that for any $t\in (0,1]$ and $Z,Z^\prime\in \RR^{d}$, 
\begin{multline*}
\sup_{|\alpha|+|\alpha^\prime|\leq m}\Bigg|\frac{\partial^{|\alpha|+|\alpha^\prime|}}{\partial Z^\alpha\partial Z^{\prime\alpha^\prime}}\Bigg(K_{\varphi(\mathcal H_{t})}(Z,Z^\prime)
-\sum_{r=0}^jF_{r}(Z,Z^\prime)t^r\Bigg)\Bigg|_{\mathcal C_b^{m^\prime}(X)}  \\
\leq Ct^{j+1}(1+|Z|+|Z^\prime|)^{M}(1+|Z-Z^\prime|)^{-N}. 
\end{multline*}
\end{thm}

\begin{proof}
In the case $m^\prime=0$, the statement follows immediately from the Taylor formula 
\[
\varphi(\mathcal H_{t})-\sum_{r=0}^j\frac{1}{r!}F_{r} t^r=\frac{1}{j!}\int_0^t(t-\tau)^j\frac{\partial^{j+1}\varphi(\mathcal H_{t})}{\partial t^{j+1}}(\tau) d\tau, \quad t\in [0,1],  
\]
and Theorem \ref{est-rem-pointwise}.

To treat the case $m^\prime \geq 1$, we proceed as in the proof of \cite[Theorem 1.10]{ma-ma08}. Let us consider the operator $\mathcal H_t$ as an operator in $C^\infty(T_{x_0}X, E_{x_0})$ given by \eqref{scaling}. 
By differentiating the formula \eqref{e:HS1} with respect to $x_0$, for any $U\in T_{x_0}X$, we get
\[
\nabla_U\varphi(\mathcal H_{t})=-\frac{1}{\pi }
\int_\CC \frac{\partial \tilde{\varphi}}{\partial \bar \lambda}(a-\lambda)^{K} \nabla_U[(\lambda-\mathcal H_{t})^{-1}(a-\mathcal H_{t})^{-K}] d\mu d\nu. 
\]
The operator $\nabla_U[(\lambda-\mathcal H_{t})^{-1}(a-\mathcal H_{t})^{-K}]$ is given by a formula similar to \eqref{diff}. Finally, the operator $\nabla_U \mathcal H_{t}$ is a differential operator on $T_{x_0}X$ of the same structure as $\mathcal H_t$ (cf. \eqref{e:partial-rLt}). This allows us to extend all our considerations to the case of an arbitrary $m^\prime  \geq 1$. 
\end{proof}

By \eqref{scaling}, we have
\[
K_{\varphi(H^{(x_0)}_p)}(Z,Z^\prime)=t^{-d}\kappa^{-\frac 12}(Z)K_{\varphi(\mathcal H_{t})}(Z/t,Z^\prime/t)\kappa^{-\frac 12}(Z^\prime), \quad Z,Z^\prime \in \mathbb R^{d},
\]
that completes the proof of the asymptotic expansion \eqref{e:main-exp} in Theorem~\ref{t:main}.

\section{Computation of the coefficients}\label{s:computation}
In this section, we derive explicit formulas for the leading coefficients in asymptotic expansions of Theorem~\ref{t:main} and Corollary~\ref{c:main}, proving   Theorem \ref{t:leading-coefficient} and the formulas \eqref{e:f0-2n} and \eqref{e:f0-d}, and describe the structure of the  lower order coefficients. 
  
Since $\lim_{t\to 0}\frac{\partial^{j}\mathcal H_{t}}{\partial t^{j}}=j!\mathcal H^{(j)}$, from \eqref{diff}, we get
\begin{multline}\label{diff0}
\lim_{t\to 0}\frac{\partial^r}{\partial t^r}(\lambda-\mathcal H_{t})^{-1}\\ =\sum_{{\mathbf r}\in I_{r}} {r}!(\lambda - \mathcal H^{(0)})^{-1} \mathcal H^{(r_1)}(\lambda - \mathcal H^{(0)})^{-1}\cdots  \mathcal H^{(r_j)}(\lambda - \mathcal H^{(0)})^{-1}.
\end{multline}
By \eqref{e:diff-phi} with $K=0$ and \eqref{e:defFr}, we infer that
\begin{multline}\label{e:Fr-phi}
F_r= -\frac{1}{\pi }\sum_{{\mathbf r}\in I_{r}}{r}! \int_\CC \frac{\partial \tilde{\varphi}}{\partial \bar \lambda}(\lambda) (\lambda - \mathcal H^{(0)})^{-1}\\ \times \mathcal H^{(r_1)}(\lambda - \mathcal H^{(0)})^{-1}\cdots  \mathcal H^{(r_j)}(\lambda - \mathcal H^{(0)})^{-1} d\mu d\nu,
\end{multline}
where $F_r$ denotes the operator in $L^2(\RR^d,E_{x_0})$ with Schwartz kernel $F_r(Z,Z^\prime)$.

For $r=0$, we immediately get
\[
F_0= -\frac{1}{\pi }\int_\CC \frac{\partial \tilde{\varphi}}{\partial \bar \lambda}(\lambda) (\lambda - \mathcal H^{(0)})^{-1} d\mu d\nu=\varphi(\mathcal H^{(0)}),
\]
which proves Theorem~\ref{t:leading-coefficient}.

To derive the formulas \eqref{e:f0-2n} and \eqref{e:f0-d} and compute lower order coefficients, we will use the technique of creation and annihilation operators as in \cite{ma-ma08}. We choose an orthonormal base $e_j, j=1,\ldots, d$, in $T_{x_0}X$ such that 
\[
 B_{x_0}e_{2j-1}=a_je_{2j},\quad B_{x_0}e_{2j}=-a_je_{2j-1},\quad j=1,\ldots, n.  
\]
\[
B_{x_0}e_{2n+m}=0,\quad m=1,\ldots, d-2n.
\]
It should be noted that, generally speaking, we can not choose the eigenvalues $a_j$'s and the orthonormal base $e_j, j=1,\ldots, d$, depending smoothly on $x_0$. As a result, the smooth dependence of the formulas, which we get, is not obvious.

We write $T_{x_0}X\cong \RR^{2n}\times \RR^{d-2n}$ with coordinates $(u,v)$ and introduce the complex coordinates $z\in\mathbb C^{n}\cong \mathbb R^{2n}$, $z_j=u_{2j-1}+iu_{2j}, j=1,\ldots,n$. Put
\[
\frac{\partial}{\partial z_j}=\frac{1}{2}\left(\frac{\partial}{\partial u_{2j-1}}-i\frac{\partial}{\partial u_{2j}}\right), \quad \frac{\partial}{\partial \bar{z}_j}=\frac{1}{2}\left(\frac{\partial}{\partial u_{2j-1}}+i\frac{\partial}{\partial u_{2j}}\right).
\]
Define first order differential operators $b_j,b^{+}_j, j=1,\ldots,n,$ on $C^\infty(\RR^d,E_{x_0})$ by the formulas
\[
b_j= -2{\frac{\partial}{\partial z_j}}+\frac{1}{2}a_j\bar{z}_j,\quad
b^{+}_j=2{\frac{\partial}{\partial\bar{z}_j}}+\frac{1}{2}a_j z_j, \quad j=1,\ldots,n.
\]
Then $b^{+}_j$ is the formal adjoint of $b_j$ on $L^2(\RR^d,E_{x_0})$, and
\begin{equation}
\mathcal H^{(0)}=\sum_{j=1}^n b_j b^{+}_j-\sum_{m=1}^{d-2n} \frac{\partial^2}{\partial v_{m}^2}+\Lambda_0. 
\label{e:Ho=bb}
\end{equation} 
We have the commutation relations
\begin{equation}
[b_i,b^{+}_j]=b_i b^{+}_j-b^{+}_j b_i =-2a_i \delta_{i\,j},\quad 
[b_i,b_j]=[b^{+}_i,b^{+}_j]=0\, ,\label{e:com1}
\end{equation} 
and, for any polynomial $g(z,\bar{z},v)$ on $z$, $\bar{z}$ and $v$, 
\begin{equation}
[g,b_j]=  2 \frac{\partial g}{\partial z_j}, \quad  [g,b_j^+]
= - 2\frac{\partial g}{\partial \bar{z}_j}\,,\quad
[g,\frac{\partial}{\partial v_m}]=- \frac{\partial g}{\partial v_m},  \label{com-bg}
\end{equation}
By \eqref{e:Ho=bb} and \eqref{e:com1}. we have
\[
b^{+}_k\mathcal H^{(0)}=(\mathcal H^{(0)}+2a_k)b^{+}_k, \quad b_k\mathcal H^{(0)}=(\mathcal H^{(0)}-2a_k)b_k, \quad \frac{\partial}{\partial v_{m}}\mathcal H^{(0)}=\mathcal H^{(0)}\frac{\partial}{\partial v_{m}},
\]
which implies that
\begin{gather}
b^{+}_k(\lambda-\mathcal H^{(0)})^{-1}=(\lambda-2a_k-\mathcal H^{(0)})^{-1}b^{+}_k,\label{e:comm-bk-res1} \\ b_k(\lambda-\mathcal H^{(0)})^{-1}=(\lambda+2a_k-\mathcal H^{(0)})^{-1}b_k, \label{e:comm-bk-res2} \\ \frac{\partial}{\partial v_{m}}(\lambda-\mathcal H^{(0)})^{-1}=(\lambda-\mathcal H^{(0)})^{-1}\frac{\partial}{\partial v_{m}}.
\label{e:comm-bk-res3}
\end{gather}
By \eqref{e:Ho=bb} and \eqref{com-bg},  for any polynomial $g(z,\bar{z},v)$ on $z$, $\bar{z}$ and $v$ with values in $\operatorname{End}(E_{x_0})$,  we have
\begin{multline}
g\mathcal H^{(0)}=\mathcal H^{(0)}g+\sum_{j=1}^n \left(2b^{+}_j\frac{\partial g}{\partial z_j} -2b_j  \frac{\partial g}{\partial \bar{z}_j}-4\frac{\partial^2g}{\partial z_j\partial \bar{z}_j}\right)\\ +\sum_{m=1}^{d-2n} \left(2\frac{\partial }{\partial v_m}\frac{\partial g}{\partial v_m}-\frac{\partial^2g}{\partial v^2_m}\right),
\end{multline}
which implies that
\begin{multline}
(\lambda-\mathcal H^{(0)})^{-1}g= g(\lambda-\mathcal H^{(0)})^{-1}\\
\begin{aligned}
+ &\sum_{j=1}^n \Big(b^{+}_j (\lambda+2a_j-\mathcal H^{(0)})^{-1}  (-2\tfrac{\partial g}{\partial z_j}) (\lambda-\mathcal H^{(0)})^{-1}\\
& + b_j (\lambda-2a_j-\mathcal H^{(0)})^{-1}2\tfrac{\partial g}{\partial \bar{z}_j} (\lambda-\mathcal H^{(0)})^{-1}\\ & -4 (\lambda-\mathcal H^{(0)})^{-1}\tfrac{\partial^2g}{\partial z_j\partial \bar{z}_j} (\lambda-\mathcal H^{(0)})^{-1}\Big)\\
- & \sum_{m=1}^{d-2n} \Big(\tfrac{\partial }{\partial v_m} (\lambda-\mathcal H^{(0)})^{-1} 2\tfrac{\partial g}{\partial v_m}(\lambda-\mathcal H^{(0)})^{-1}\\ & - (\lambda-\mathcal H^{(0)})^{-1}\tfrac{\partial^2g}{\partial v^2_m}(\lambda-\mathcal H^{(0)})^{-1}\Big).
\end{aligned}
\label{e:comm-g-res}
\end{multline}
We will apply the formula \eqref{e:comm-g-res} to write the operator $(\lambda-\mathcal H^{(0)})^{-1}g$ in the form $D(\lambda-\mathcal H^{(0)})^{-1}$ with some differential operator $D$. We see that the first term in the right hand side of \eqref{e:comm-g-res} already has the desired form, the other terms don't, but the degrees of polynomials entering these terms decrease at least by one. Applying the formula \eqref{e:comm-g-res} to each of these terms, 
we decrease the degree of polynomials more. We can proceed further, and, after finitely many steps, we eliminate these bad terms, arriving at an expression of the form
\[
(\lambda-\mathcal H^{(0)})^{-1}g=\sum_{j\in J_r}D_{j,g}\prod_{|\mathbf k|\leq r}(\lambda+2\mathbf k\cdot a -\mathcal H^{(0)})^{-j_{\mathbf k}},
\] 
where $r$ is the degree of $g$ and the sum is taken over the set $J_r$ of all collections $j=\{j_{\mathbf k}\in \mathbb Z_+ : {\mathbf k}\in \ZZ^n\}$ such that $ |\mathbf k|\leq r$ and $\sum j_{\mathbf k}\leq r+1$ and $D_{j,g}$ is a differential operator of weight $r$. Here we introduce a grading on the algebra of differential operators on $C^\infty(\RR^d,E_{x_0})$ with polynomial coefficients, setting the weight of each $z_j$, $\bar{z}_k$ and $v_m$ to be equal to $1$ and the weight of each of operators $b_j$, $b^{+}_k$ and $\frac{\partial}{\partial v_m}$ to be equal to $1$. 

As a consequence, using \eqref{e:comm-bk-res1}, \eqref{e:comm-bk-res2} and \eqref{e:comm-bk-res3}, for any differential operator $A$ of weight $p$, we get
\[
(\lambda-\mathcal H^{(0)})^{-1}A=\sum_{j\in J_p}D_{j,A}\prod_{|\mathbf k|\leq p}(\lambda+2\mathbf k\cdot a -\mathcal H^{(0)})^{-j_{\mathbf k}},
\] 
where $D_{j,A}$ is a differential operator of weight $p$. 

By \eqref{e:Hj}, each $\mathcal H^{(r)}$ has weight $r+2$. It follows that 
\begin{multline*}
(\lambda - \mathcal H^{(0)})^{-1} \mathcal H^{(r_1)}(\lambda - \mathcal H^{(0)})^{-1}\cdots  \mathcal H^{(r_j)}(\lambda - \mathcal H^{(0)})^{-1} \\ =\sum_{j}D_{j}\prod_{|\mathbf k|\leq r+2j}(\lambda+2\mathbf k\cdot a -\mathcal H^{(0)})^{-j_{\mathbf k}},
\end{multline*} 
where the sum is taken over the set of all collections $j=\{j_{\mathbf k} : {\mathbf k}\in \ZZ^n, |\mathbf k|\leq p\}$ such that $\sum j_{\mathbf k}\leq N(\mathbf r)$ with some $N(\mathbf r)\in \mathbb N$ and $D_{j}$ is a differential operator of weight $r+2j$.

Now applying the partial fraction decomposition to each product of resolvents, we can rewrite the above formula as follows:
\begin{multline*}
(\lambda - \mathcal H^{(0)})^{-1} \mathcal H^{(r_1)}(\lambda - \mathcal H^{(0)})^{-1}\cdots  \mathcal H^{(r_j)}(\lambda - \mathcal H^{(0)})^{-1} \\ =\sum_{|\mathbf k|\leq r+2j} \sum_{\ell=1}^{N(\mathbf r)} D_{\mathbf k,\ell,\mathbf r} (\lambda+2\mathbf k\cdot a -\mathcal H^{(0)})^{-\ell},
\end{multline*} 
where $D_{\mathbf k,\ell,\mathbf r}$ is a differential operator of weight $r+2j\leq 3r$. 

By \eqref{e:Fr-phi}, we conclude that
\[
F_r= -\frac{1}{\pi }\sum_{{\mathbf r}\in I_{r}}\sum_{|\mathbf k|\leq 3r} \sum_{\ell=1}^{N(\mathbf r)} D_{\mathbf k,\ell,\mathbf r} {r}! \int_\CC \frac{\partial \tilde{\varphi}}{\partial \bar \lambda}(\lambda)(\lambda+2\mathbf k\cdot a -\mathcal H^{(0)})^{-\ell} d\mu d\nu.
\]
By differentiating the Helffer-Sjostrand formula \eqref{e:HS}, we get  
\[
\varphi^{(k)}(\mathcal H^{(0)})=-\frac{k!}{\pi }
\int_\CC \frac{\partial \tilde{\varphi}}{\partial \bar \lambda}(\lambda)(\lambda-\mathcal H^{(0)})^{-k-1}d\mu d\nu.
\]
It follows that
\[
F_r=\sum_{|\mathbf k|\leq 3r} \sum_{\ell=1}^{N_r} D_{\mathbf k,\ell} \varphi^{(\ell-1)}(\mathcal H^{(0)}-2\mathbf k\cdot a),
\]
where $D_{\mathbf k,\ell}$ is a differential operator of weight $3r$ and $N_r=\max_{\mathbf r\in I_r} N(\mathbf r)$. 

In particular, we have
\begin{multline}\label{e:fr2}
f_r(x_0)= \operatorname{tr} F_r(0,0)\\ =\sum_{|\mathbf k|\leq r} \sum_{\ell=1}^{N_r}\sum_{|\alpha|\leq 3r}\operatorname{tr} c_{\mathbf k,\ell.\alpha}(x_0) \nabla^\alpha_Z K_{\varphi^{(\ell-1)}(\mathcal H^{(0)}-2\mathbf k\cdot a)}(0,0),
\end{multline}
where $c_{\mathbf k,\ell.\alpha}(x_0)\in \operatorname{End}(E_{x_0})$ and $K_{\varphi(\mathcal H^{(0)})}(Z,Z^\prime)$ denotes the smooth kernel of the operator $\varphi(\mathcal H^{(0)})$.

Denote by $\pi_{\mu,x_0}\in \operatorname{End}(E_{x_0})$ the spectral projection of the self-adjoint operator $V(x_0)$, corresponding to the eigenvalue $V_\mu(x_0), \mu=1,\ldots,\operatorname{rank}(E)$. Using separation of variables and the Fourier transform, we infer that
\begin{multline}\label{e:Phi-kernel}
K_{\varphi (\mathcal H^{(0)})}(Z,Z^\prime)\\ =\frac{1}{(2\pi)^{d-2n}} \sum_{\mathbf k \in\ZZ^n_+}\sum_{\mu=1}^{\operatorname{rank}(E)} \mathcal P_{\Lambda_{\mathbf k}}(u,u^\prime)\pi_{\mu,x_0} \int_{\RR^{d-2n}} e^{i(v-v^\prime)\xi}\varphi(\Lambda_{\mathbf k,\mu}+|\xi|^2)d\xi,
\end{multline}
where $\mathcal P_{\Lambda_{\mathbf k}}(u,u^\prime)$ is the Schwartz kernel of the spectral projection of the operator $\mathcal L_0=\sum_{j=1}^n b_j b^{+}_j$ in $L^2(\mathbb R^{2n})$ corresponding to the eigenvalue $\Lambda_{\mathbf k}=\sum_{j=1}^n(2k_j+1) a_j$.   

By \eqref{e:Phi-kernel},  taking into account that $\mathcal P_{\Lambda_{\mathbf k}}(0,0)= \frac{1}{(2\pi)^n}\prod_{j=1}^na_j$,
for $d=2n$, we get
\[
K_{\varphi (\mathcal H^{(0)})}(Z,Z)=\frac{1}{(2\pi)^n}\prod_{j=1}^na_j(x_0) \sum_{\mathbf k \in\ZZ^n_+}\sum_{\mu=1}^{\operatorname{rank}(E)} \varphi(\Lambda_{\mathbf k,\mu}(x_0))\pi_{\mu,x_0},
\]
which gives \eqref{e:f0-2n} by \eqref{e:f0}, and, for $d>2n$, 
\begin{multline*}
K_{\varphi (\mathcal H^{(0)})}(Z,Z)\\ =\frac{1}{(2\pi)^{d-n}}\prod_{j=1}^na_j(x_0) \sum_{\mathbf k \in\ZZ^n_+}\sum_{\mu=1}^{\operatorname{rank}(E)} \int_{\RR^{d-2n}} \varphi(\Lambda_{\mathbf k,\mu}(x_0)+|\xi|^2)d\xi \pi_{\mu,x_0},
%=& \frac{|S^{d-2n-1}|}{2(2\pi)^{d-n}}\prod_{j=1}^na_j(x_0) \sum_{\mathbf k \in\ZZ^n_+}\sum_{\mu=1}^{\operatorname{rank}(E)} \int_0^{+\infty} \varphi(\tau)(\tau-\Lambda_{\mathbf k,\mu}(x_0))_+^{d/2-n-1}d\tau \pi_{\mu,x_0},
\end{multline*}
which gives \eqref{e:f0-d} by \eqref{e:f0}.

For an arbitrary $r$, by \eqref{e:Phi-kernel} and \eqref{e:fr2}, for $d=2n$, we get
\begin{equation}\label{e:fr1-form}
f_r(x_0)=\sum_{\mathbf k \in\ZZ^n_+}\sum_{\mu=1}^{\operatorname{rank}(E)} \sum_{\ell=1}^{N} P_{\mathbf k,\mu,\ell}(x_0) \varphi^{(\ell-1)}(\Lambda_{\mathbf k,\mu}(x_0)),
\end{equation}
where $P_{\mathbf k,\mu,\ell}$ has power growth in $\mathbf k$, and, for $d>2n$, we get
\begin{equation}\label{e:fr2-form}
f_r(x_0)=\sum_{\mathbf k \in\ZZ^n_+}\sum_{\mu=1}^{\operatorname{rank}(E)} \sum_{\ell=1}^{N} \int_{\RR^{d-2n}} P_{\mathbf k,\mu,\ell,x_0}(\xi) \varphi^{(\ell-1)}(\Lambda_{\mathbf k,\mu}(x_0)+|\xi|^2)d\xi,
\end{equation}
where $P_{\mathbf k,\mu,\ell,x_0}(\xi)$ is a polynomial of degree $3r$. %Introducing polar coordinates, we get
%\[
%f_r(x_0)=\sum_{\mathbf k \in\ZZ^n_+} \sum_{\ell=1}^{N}
%\int_0^{+\infty} \varphi^{(\ell-1)}(\tau)p_{\mathbf k,\ell,x_0}(\tau-\Lambda_{\mathbf k,\mu}(x_0))(\tau-\Lambda_{\mathbf k,\mu}(x_0))_+^{d/2-n-1} d\tau,
%\]
%where $p_{\mathbf k,\ell,x_0}$ is a polynomial of degree $3r$ given by
%\[
%p_{\mathbf k,\ell,x_0}(\tau)=\frac{1}{2}|S^{d-2n-1}|\int_{|\xi|=\tau} P_{\mathbf k,\ell,x_0}(\xi)d\xi, \quad \tau\geq 0.
%\]

Now we are ready to prove Theorem~\ref{t:local}. Assume that the rank of $B_{x_0}$ equals $d$ and an interval $(\alpha,\beta)$ does not contains any $\Lambda_{\mathbf k,\mu}(x_0)$ with  $\mathbf k \in \mathbb Z^n_+$ and $\mu=1,\ldots,\operatorname{rank}(E)$ and $\varphi\in \mathcal S(\RR)$ is supported in $(\alpha,\beta)$. By \eqref{e:fr1-form}, we infer that $f_r(x_0)=0$ for any $r=0,1,\ldots,$ that proves \eqref{e:loc1}. To prove \eqref{e:loc2}, given an interval $[\alpha,\beta]$ does not contains any $\Lambda_{\mathbf k,\mu}(x_0)$ with  $\mathbf k \in \mathbb Z^n_+$ and $\mu=1,\ldots,\operatorname{rank}(E)$, we choose a function $\varphi\in \mathcal S(\RR)$, supported in an interval $(\alpha_1,\beta_1)$, which contains $[\alpha,\beta]$ and does not contains any $\Lambda_{\mathbf k,\mu}(x_0)$ with  $\mathbf k \in \mathbb Z^n_+$ and $\mu=1,\ldots,\operatorname{rank}(E)$, such that $\chi_{[\alpha,\beta]}(x)\leq \varphi(x)$ for any $x\in \mathbb R$. Then 
\[
E_{[\alpha,\beta]}=\chi_{[\alpha,\beta]}(H_p)\leq \varphi(H_p)
\] 
and 
\[
E_{[\alpha,\beta]}(x,x)\leq K_{\varphi(H_p)}(x,x)
\]
for any $x\in X$. By \eqref{e:loc1}, this gives \eqref{e:loc2}. 

%To prove \eqref{e:loc2} for an arbitrary $k$, we write 
%\[
%(H_p+c)^{N}E_{[\alpha,\beta]}(H_p+c)^{N} \leq (H_p+c)^{N}\varphi(H_p)(H_p+c)^{N}
%\]
%with any $N\in \mathbb N$, where $c$ is a sufficienly large constant such thet $H_p+c>0$ for any $p$, and accordingly 
%\[
%(H_{p,x}+c)^{N}(H_{p,y}+c)^{N}E_{[\alpha,\beta]}(x,y)\left|_{y=x}\right. \leq K_{(H_p+c)^{2N}\varphi(H_p)}(x,x)
%\]
%
%then the Schwartz kernel of the spectral projection $E_{[\alpha,\beta]}$ of the operator $H_p$ associated with $[\alpha,\beta]$ satisfies
%\[
%\left|E_{[\alpha,\beta]}(x_0,x_0)\right|_{{C}^k}=\mathcal O(p^{-\infty}), \quad p\to \infty.
%\]


\begin{thebibliography}{00}
\bibitem{bismut87}
J.-M. Bismut,  Demailly's asymptotic Morse inequalities: a heat equation proof. J. Funct. Anal. \textbf{72} (1987), 263--278.

\bibitem{BL} 
J.-M. Bismut,  G. Lebeau, Complex immersions and Quillen metrics.
Inst. Hautes \'Etudes Sci.\ Publ.\ Math. \textbf{74} (1991). 

\bibitem{B-Uribe07} 
D. Borthwick, A. Uribe, The semiclassical structure of low-energy states in the presence of a magnetic field. Trans. Amer. Math. Soc. \textbf{359} (2007),  1875--1888.

\bibitem{bouche90}
Th. Bouche, Convergence de la m\'{e}trique de {F}ubini-{S}tudy d'un fibr\'{e} lin\'{e}aire positif. {Ann. Inst. Fourier} \textbf{49} (1990), 117--130.

\bibitem{BPU95}
R. Brummelhuis, T. Paul, A. Uribe, Spectral estimate near a critical level. Duke Math. J. \textbf{78} (1995) 477--530. 

\bibitem{camus04}
B. Camus, A semi-classical trace formula at a non-degenerate critical level,
J. Funct. Anal. \textbf{208} (2004), 446--481. 

\bibitem{charles21}
L. Charles, On the spectrum of non degenerate magnetic Laplacian. Preprint arXiv:2109.05508, 2021.

\bibitem{dai-liu-ma}
X.~Dai, K.~Liu, X.~Ma, 
{On the asymptotic expansion of {B}ergman
  kernel}, J.\ Differential Geom.\ \textbf{72} (2006), 1--41.

\bibitem{Demailly85}
J.-P. Demailly, Champs magn\'{e}tiques et in\'{e}galit\'{e}s de {M}orse pour la {$d''$}-cohomologie. Ann. Inst. Fourier (Grenoble) \textbf{35} (1985), no. 4, 189--229.

\bibitem{Demailly91}
J.-P. Demailly, Holomorphic Morse inequalities. In \textit{Several complex variables and complex geometry, Part 2 (Santa Cruz, CA, 1989)}, 93--114, Proc. Sympos. Pure Math., 52, Part 2, Amer. Math. Soc., Providence, RI, 1991.

\bibitem{FT}
F. Faure, M. Tsujii, {Prequantum transfer operator for symplectic Anosov diffeomorphism.} Ast\'{e}risque No. \textbf{375} (2015)

\bibitem{Gu-Uribe}
V. Guillemin, A. Uribe,  The Laplace operator on the $n$th tensor power of a line bundle: eigenvalues which are uniformly bounded in $n$. Asymptotic Anal. \textbf{1} (1988), 105--113.

\bibitem{HS-LNP345}
B. Helffer, J. Sj\"{o}strand, \'{E}quation de {S}chr\"{o}dinger avec champ magn\'{e}tique et \'{e}quation de
              {H}arper. Schr\"{o}dinger operators ({S}\o nderborg, 1988), 118--197, Lecture Notes in Phys., 345, Springer, Berlin, 1989

\bibitem{K-D97}
D. Khuat-Duy, A semi-classical trace formula at a critical level. J. Funct. Anal. \textbf{146} (1997), 299--351.

\bibitem{Kor18}
Yu. A. Kordyukov, On asymptotic expansions of generalized Bergman kernels on symplectic manifolds; translated from Algebra i Analiz \textbf{30} (2018), no. 2, 163--187 St. Petersburg Math. J. \textbf{30} (2019), no. 2, 267--283

\bibitem{higherLL} Yu. A. Kordyukov, Semiclassical spectral analysis of the Bochner-Schr\"odinger operator on symplectic manifolds of bounded geometry, Anal. Math. Phys. \textbf{12} (2022), no. 1, Paper No. 22, 37 pp.

\bibitem{heat} Yu. A. Kordyukov, Heat kernels estimates for Hermitian line bundles on manifolds of bounded geometry, Mathematics, \textbf{9(23)} (2021),  3060.

\bibitem{UMN-trace} Yu. A. Kordyukov, Trace formula for the magnetic Laplacian at zero energy level, preprint arXiv:2208.04599, 2022. 

\bibitem{ko-ma-ma} Yu. A. Kordyukov, X. Ma, G. Marinescu, 
Generalized Bergman kernels on symplectic manifolds of bounded geometry. Comm. Partial Differential Equations \textbf{44} (2019), 1037--1071.

\bibitem{ma-ma:book}
X. Ma,  G.~Marinescu, Holomorphic Morse inequalities and Bergman kernels. Progress in Mathematics, 254. Birkh\"auser Verlag, Basel, 2007. 

\bibitem{ma-ma08} 
X. Ma, G.~Marinescu, Generalized Bergman kernels on symplectic manifolds. Adv. Math. \textbf{217} (2008), 1756--1815.
%
%\bibitem{ma-ma15} X.~Ma and G.~Marinescu, 
%\emph{Exponential estimate for the asymptotics of Bergman kernels}. 
%Math. Ann. \textbf{362} (2015), no. 3-4, 1327--1347.

\bibitem{ma-ma-zelditch15} 
X. Ma, G.~Marinescu, S. Zelditch, Scaling asymptotics of heat kernels of line bundles. Analysis, complex geometry, and mathematical physics: in honor of Duong H. Phong; Contemp. Math., 644, Amer. Math. Soc., Providence, RI, 2015; pp. 175--202.

\bibitem{Marinescu-Savale} G.~Marinescu, N. Savale, Bochner Laplacian and Bergman kernel expansion of semi-positive line bundles on a Riemann surface, preprint arXiv:1811.00992 

\bibitem{Morin19}
L. Morin, A semiclassical Birkhoff normal form for symplectic magnetic wells. J. Spectral Theory, 2022 (to appear); preprint arXiv:1907.03493 (2019)

\bibitem{Savale17}
N. Savale, Koszul complexes, Birkhoff normal form and the magnetic Dirac operator. Anal. PDE \textbf{10} (2017), 1793--1844.
\end{thebibliography}
\end{document}